\documentclass[onefignum,onetabnum]{siamart190516}

\usepackage{amsmath,amsfonts,amssymb}
\usepackage{graphicx}
\usepackage{epstopdf}
\usepackage{algorithmic}
\usepackage{hyperref}
\ifpdf
  \DeclareGraphicsExtensions{.eps,.pdf,.png,.jpg}
\else
  \DeclareGraphicsExtensions{.eps}
\fi

\usepackage{tikz}
\usetikzlibrary{calc,patterns}
\tikzstyle{bag} = [align=center]

\usepackage{caption}
\usepackage{subcaption}

\newsiamremark{remark}{Remark}
\newsiamremark{hypothesis}{Hypothesis}
\crefname{hypothesis}{Hypothesis}{Hypotheses}
\newsiamthm{claim}{Claim}

\headers{Derivative-Informed Neural Networks for Parametric Maps}{T. O'Leary-Roseberry, U. Villa, P. Chen, O. Ghattas}

 \title{Derivative-Informed Projected Neural Networks for High-Dimensional
   Parametric Maps Governed by PDEs\thanks{This research partially
     supported by ARPA-E DIFFERENTIATE grant DE-AR0001208; NSF
     SI2-SSI grant ACI-1550593; DOE ASCR MMICC grant DE-SC0019303; and
     NSF DMS grant 2012453.}} 

\author{Thomas~O'Leary-Roseberry\thanks{Oden Institute for Computational Engineering \& Sciences,
  The University of Texas at Austin, Austin, TX
  (\email{tom@oden.utexas.edu}, \email{peng@oden.utexas.edu}).}
\and Umberto~Villa\thanks{McKelvey School of Engineering, Washington University in St. Louis, Saint Louis, MO
  (\email{uvilla@wustl.edu}).}
\and Peng~Chen\footnotemark[1] \and 
 Omar~Ghattas\thanks{Oden Institute for Computational Engineering \& Sciences,
  Jackson School of Geosciences, and Department of Mechanical Engineering,
  The University of Texas at Austin, Austin, TX 
  (\email{omar@oden.utexas.edu}).}}

\usepackage{amsopn}

\begin{document}

\maketitle

\begin{abstract}
Many-query problems---arising from, e.g., uncertainty quantification,
Bayesian inversion, Bayesian optimal experimental design, and
optimization under uncertainty---require numerous evaluations of a
parameter-to-output map. These evaluations become prohibitive if this
parametric map is high-dimensional and involves expensive solution of
partial differential equations (PDEs). To tackle this challenge, we
propose to construct surrogates for high-dimensional PDE-governed
parametric maps in the form of derivative-informed projected neural networks (DIPNets) that
parsimoniously capture the geometry and intrinsic low-dimensionality
of these maps.
Specifically, we compute Jacobians of these PDE-based maps, and project the
high-dimensional parameters onto a low-dimensional derivative-informed
active subspace; we also project the possibly high-dimensional outputs
onto their principal subspace. This exploits the fact that many
high-dimensional PDE-governed parametric maps can be
well-approximated in low-dimensional parameter and output subspaces.
We use the projection basis vectors in the active subspace as well as
the principal output subspace to construct the weights for the first
and last layers of the neural network, respectively.  This frees us to
train the weights in only the low-dimensional layers of the
neural network.  
The architecture of the resulting neural network then captures, to
first order, the low-dimensional structure and geometry of the
parametric map.
We demonstrate that the proposed projected neural network achieves
greater generalization accuracy than a full neural network, especially
in the limited training data regime afforded by expensive PDE-based
parametric maps.
Moreover, we
show that the number of degrees of freedom of the inner layers of
the projected network is independent of the parameter and output 
dimensions, and high accuracy can be achieved with weight dimension
independent of the discretization dimension.
\end{abstract}



\section{Introduction}
Many problems in computational science and engineering require
repeated evaluation of an expensive nonlinear parametric map for
numerous instances of input parameters drawn from a probability
distribution $\nu$. These {\em many-query problems} arise in such
problems as Bayesian inference, forward uncertainty quantification,
optimization under uncertainty, and Bayesian optimal experimental
design (OED), and are often governed by partial
differential equations (PDEs). 
The maps are parameterized by model parameters $m$ with joint
probability distribution $\nu$, which is mapped to outputs $q$ through an implicit dependence on the solution of
the PDEs for the state $u$: 
\begin{equation}
  \underbrace{q(m) = q(u(m))}_\text{Implicit dependence} \text{ where } u \text{ depends on } m \text{ through } \underbrace{R(u,m) = 0}_\text{PDE model}
\end{equation}

The variables $m\in \mathbb{R}^{d_M},u\in \mathbb{R}^{d_U}$ are
formally discretizations of functions in Banach spaces, where the
dimensions $d_M$ and $d_U$ are often large. We assume that the outputs
$q \in \mathbb{R}^{d_Q}$ are differentiable with respect to $m$. The
outputs can be the full PDE states $u$, or state dependent quantity of interest, such as
integrals or (mollified) pointwise evaluations of the states or their
fluxes, as arise in inverse problems, optimal design and control, and
OED.

Each evaluation of $q$ requires solution of the PDEs at one instance
of $m$, thus making solution of many-query problems prohibitive when
the solution of the governing PDEs is expensive due to nonlinearity,
multiphysics or multiscale nature, and/or complex geometry. For
example, solution of complex PDEs can take hours or even days or
weeks on a supercomputer, making the solution of many-query problems
prohibitive using the full high fidelity PDE solution.

Thus, practical solution of high dimensional many-query problems
requires surrogates for the mapping $m \mapsto q$. The goal of surrogate
construction is to construct approximations of the map that are
accurate over the probability distribution for input parameters, $\nu$,
i.e., to find a surrogate $f \in \mathbb{R}^{d_Q}$ that is inexpensive
to construct and evaluate and that satisfies 
\begin{equation}\label{expected_error_bound}
  \mathbb{E}_\nu [\| q - f\|^2] = \int \|q(m) - f(m)\|^2 \, d\nu(m)  < \epsilon
\end{equation}
for a suitable tolerance $\epsilon >0$, in a suitable norm. Here $\mathbb{E}_\nu$ is the expectation with respect to $\nu$, as defined in equation \eqref{expected_error_bound}.

To tackle this challenge, various forms of surrogates have been
developed that exploit intrinsic properties of the high-dimensional
parameter-to-output map, $m \mapsto q$ such as sparsity and
low-dimensionality, including, e.g., sparse polynomials
\cite{ BabuskaNobileTempone10,ChenQuarteroni15, CohenDeVore15,
  XiuKarniadakis02a}, and reduced basis methods
\cite{Bui-ThanhWillcoxGhattas08a, ChenQuarteroniRozza2017,ChenSchwab2016,
  CohenDeVore15}. Dimension-independent complexity to
achieve a desired accuracy has been shown for these methods under
suitable assumptions; see the review in \cite{CohenDeVore15} and
references therein. In recent years, neural networks have shown great
promise in representing high dimensional nonlinear mappings,
especially in data-rich applications such as computer vision and
natural language processing; see for example
\cite{GoodfellowBengioCourville2016}. In these 
regimes neural networks with high weight dimensions can learn patterns
in very large data sets. More recently, significant interest has
arisen in using neural networks as nonlinear surrogates for
PDE-governed high-dimensional parametric maps in the many-query
setting \cite{BhattacharyaHosseiniKovachki2020,HanJentzenWeinan18, 
  RuthottoOsherLiEtAl20, SchwabZech19,
  ZhuZabarasKoutsourelakisEtAl19}.

However, the PDE-governed many-query regime is fundamentally different
from the data-rich regime of typical data science problems. The key
difference is that in the former, few data may be available to train a
surrogate due to the expense of evaluating the map, and the input--output
pairs may be very high dimensional. For example, consider a high
parameter dimension climate model: the parameters (e.g., unknown
initial conditions) and outputs (e.g., future states) may number in the
millions or billions, but one can afford only hundreds of evaluations
of the mapping $m \mapsto q$. Because the number of weights in a dense
neural network model would then
scale with the square of the number of parameters/outputs, the number
of climate model runs needed to fully inform the weights would need to
scale with the number of parameters and outputs. This is orders of
magnitude larger than is feasible.
Black box strategies for dimension reduction such as convolution
kernels can efficiently represent discretizations of stationary
processes on regular (Cartesian) grids. The limitation to regular
grids presents difficulties for many PDE problems discretized on
highly unstructured meshes. Recent works in geometric deep learning have extended convolution processes to non-Cartesian data \cite{BronsteinBrunaLeCunEtAl17,MontiBoscainiMasciEtAl17}, but even when the meshes are
uniform, the question of finding the right architecture remains, as does
the question of associated dimension reduction. Genetic algorithms can be used
to select neural network architectures
\cite{Branke1995,SuganumaShirakawaNagao2017,Yao1993}, but this can be
prohibitive, since each proposed network needs to be trained in order
to test its viability, and the number of trained networks that need to
be explored will be large due to the heuristic nature of these
algorithms.


The key idea for overcoming this critical challenge faced by the PDE-governed
setting is that the mapping $m \mapsto q$ 
often has low-dimensional structure and smoothness that can be exploited to 
aid architecture selection  as well as
provide informed subspaces of the input and intrinsic subspaces of the
outputs, within which the mapping can be well approximated.
This intrinsic low-dimensionality is revealed by the Jacobian
and higher-order derivatives of $q$ with respect to $m$.  This has
been proven for certain model problems and demonstrated numerically
for many others in model reduction \cite{AlgerChenGhattas20,BashirWillcoxGhattasEtAl08,
  ChenGhattas2019}, Bayesian inversion
\cite{Bui-ThanhBursteddeGhattasEtAl12_gbfinalist, Bui-ThanhGhattas12a,
  Bui-ThanhGhattas12, Bui-ThanhGhattas13a, Bui-ThanhGhattas15,
  Bui-ThanhGhattasMartinEtAl13,ChenGhattas20,ChenVillaGhattas2017, ChenWuChenEtAl19, FlathWilcoxAkcelikEtAl11,
  IsaacPetraStadlerEtAl15, MartinWilcoxBursteddeEtAl12,PetraMartinStadlerEtAl14a}, optimization
under uncertainty \cite{AlexanderianPetraStadlerEtAl17,ChenGhattas20a, ChenHabermanGhattas21,
  ChenVillaGhattas19}, and
Bayesian optimal experimental design
\cite{AlexanderianPetraStadlerEtAl14, AlexanderianPetraStadlerEtAl16,
  CrestelAlexanderianStadlerEtAl17, WuChenGhattas20}.

In this work we consider strategies for constructing reduced dimension
projected neural network surrogates using Jacobian-based derivative
information for effective input dimension reduction, and represent the outputs in their principal subspace. The resulting networks
can be viewed as encoder-decoder networks with a priori computed
reduced input and output bases that reflect the geometry and
low-dimensionality of the parameter-to-output map.  When the bases for
the inputs and outputs are fixed, the weight dimension of the neural
network is independent of the discretization dimensions of the
PDE model. On the other hand, when these bases are permitted to be
modified during neural network training, they serve as good initial
guesses that impose the parametric map's structure on the full
encoder-decoder training problem.

We consider active subspaces (AS) \cite{ConstantineDowWang2014,
  ZahmConstantinePrieurEtAl2020} for the input space reduction, and the principal subspace given by
proper orthogonal decomposition (POD)
\cite{ManzoniNegriQuarteroni2016, QuarteroniManzoniNegri2015} for the
output reduction. AS incorporates Jacobian information to detect input
subspaces to which the outputs are sensitive, while POD provides a low dimensional
subspace in  which the outputs are efficiently represented, we call this neural network strategy DIPNet (derivative informed projected neural network).
We compare this strategy to the use of Karhunen Lo\`{e}ve expansion (KLE)
projections for the input space and POD for the output projection, as
was done in \cite{BhattacharyaHosseiniKovachki2020}. KLE exploits low
dimensionality of the parameter probability distribution, as exposed
by the eigenvalue decay of its covariance operator, but does not
account for the outputs. KLE and POD are related to principal component
analysis (PCA) of the parameter data and output data respectively. PCA
is a popular strategy for dimension reduction in neural network
architecture selection; see for example \cite{GargPandaRoy2019}. We
favor the AS for the input projection, since it is derivative-informed; AS
explicitly incorporates the sensitivity of the outputs to the
inputs in determining the dimension reduction. 

To motivate this strategy, we derive an input--output projection error
bound based on optimal ridge functions for Gaussian parametric
mappings. We construct DIPNet surrogates based
on AS-to-POD neural network ridge functions, which we compare against
KLE-to-POD neural network ridge functions as well as conventional
approaches using full space neural networks.

We test the resulting surrogates on two different PDE
parameter-to-output map problems: one parametric mapping involving
pointwise observations of the state for a nonlinear
convection-diffusion-reaction problem, and another involving a high
frequency Helmholtz problem. We consider problems where $d_M$ is
large, but $d_Q$ is smaller; thus, we keep the first layer fixed but
train the output layer. These numerical experiments demonstrate that
full space neural network surrogates that depend on the dimension of
the discretization parameter perform poorly in generalization accuracy
in the low data regime. On the other hand, the projected neural
network surrogates are capable of achieving high accuracy in the low
data regime; in particular the DIPNet strategy performs 
best. We also test against neural network ridge functions with
identical architectures that instead use Gaussian random projection
bases to test the effect of the structured bases (DIPNet and
KLE-to-POD), and demonstrate that the random projections performs
worse than the structured ones. 

\textbf{Contributions:} We propose a general framework to construct
surrogates for high-dimensional parameter-to-output maps governed by
PDEs in the form of \linebreak derivative-informed low-dimensional projected neural
networks (DIPNets) that generalize well even with limited numbers of training
data. This framework exploits informed  input and output subspaces
that reveal the intrinsic dimensionality of the parameter-to-output 
map. The strategy involves constructing ridge function surrogates that
attempt to learn only the nonlinear mapping between informed modes of
the input parameters and output spaces. This strategy is infinite-dimensionally 
consistent since the operators used in the neural network construction are 
mesh independent.

Our work differs from \cite{BhattacharyaHosseiniKovachki2020},
which uses KLE-based input projections; instead, 
we use derivative-informed active subspace-based projections.  Both
representations have input--output projection error bounds given
respectively by the eigenvalue decays of the AS, KLE, and POD
operators. For a fixed rank ridge function surrogate, the bounds for
the AS-to-POD surrogates are tighter than those for the KLE-to-POD,
which involve the square of the Lipschitz constant for the
parameter-to-output map. These are conservative upper bounds, but one
can reasonably expect that incorporating specific information about
the outputs in the input basis can provide better accuracy with fewer
modes. This is confirmed in numerical experiments. 

Numerical experiments demonstrate that the DIPNet strategy
using AS performs well in the limited data regime typical of
PDE-governed parametric maps. When the squared Jacobian of the
parameter-to-output map has spectral structure similar to the
covariance operator of the input parameter distribution, the KLE
eigenvectors perform well as a reduced basis---nearly as well as the
AS strategy. When the parameter-to-output map differs from the
covariance operator, the KLE strategy is seen to perform worse than
the AS strategy, and only slightly better than using Gaussian random
bases.

The rest of the paper is organized as follows: In Section 2, we discuss techniques for input-output dimension reduction for parametric maps, and derive an input--output projection error
bound based on optimal ridge functions for Gaussian parametric
mappings based on AS, KLE and POD. In Section 3 we present strategies for constructing DIPNet surrogates using these building blocks, and discuss errors incurred in the surrogate construction procedure. In Section 4 we present numerical experiments that demonstrate the viability of the derivative-informed projected neural network approach for two problems arising from parametric PDE inference problems.

\section{Input-Output Projected Ridge Function Surrogates}

We proceed by discussing strategies for input-output projected ridge functions of the form
\begin{equation} \label{input_output_ridge_function}
  q(m) \approx f(m) = \Phi_{r_Q}f_r(V_{r_M}^Tm) + b_Q,
\end{equation}
where $V_{r_M}$ are a reduced basis for the input space, $\Phi_{r_Q}$ are a reduced basis for the outputs, and $f_r:\mathbb{R}^{r_M}\rightarrow \mathbb{R}^{d_Q}$ is a ridge function. We begin by examining strategies for input and output projection, and known bounds for projection errors, which lead to an input--output bound that motivates our general approach.

\subsection{Input Projection: Active Subspace and Karhunen-Lo\`{e}ve Expansion}

In the setting of parametric mappings, the input parameters $m$ come from a discretization of an infinite dimensional parameters. Algorithmic complexity suffers from the curse of dimensionality as the discretization dimension grows. In order to avoid the curse of dimensionality we seek to exploit low-dimensional structure of the parameter-to-output map in architecting neural network surrogates. 

An effective strategy for constructing an input dimension reduced surrogate is to construct an optimal ridge function approximation of $m \mapsto q$. A ridge function is a composition of the form $g \circ h$, where $h:\mathbb{R}^{d_M}\rightarrow \mathbb{R}^r$ is a linear mapping (a tall skinny matrix, with $r \ll d_M$), and $g:\mathbb{R}^r \rightarrow \mathbb{R}^{d_Q}$ is a measurable function in $L^2(\mathbb{R}^{d_M}, \sigma(h),\mathbb{R}^{d_Q})$, i.e. in the $\sigma$-algebra generated by the bases of $h$. Input dimension reduced dense neural networks have this form.

We seek to construct ridge function surrogates for the mapping $m \mapsto q$ that are restricted to subspaces of the input space that resolve intrinsic low dimensionality about the mapping. The goal is to find a (nonlinear) ridge function $g$ and a linear projector of dimension $r$, $V_r:\mathbb{R}^{d_M} \rightarrow \mathbb{R}^r$ such that for a given tolerance $\epsilon >0$,
\begin{equation}
	\mathbb{E}_\nu [\| q - g \circ V_r\|_X] \leq \epsilon
\end{equation}

in a suitable norm $X$. The key issues that remain are how to find $V_r$, i.e. both the dimension $r$ and the subspace $X_r \subset \mathbb{R}^{d_M}$ spanned by the column bases of $V_r$.

One approach we consider is to use Karhunen-Lo\`{e}ve expansion (KLE) \cite{SchwabTodor2006}, which exploits low dimensional correlation structure of the probability distribution $\nu$ of the parameters $m$. This approach is equivalent to principal component analysis (PCA) of input data. It is considered for solving parametric PDEs using dimension reduced neural networks in \cite{BhattacharyaHosseiniKovachki2020}.

For a Gaussian distribution $\nu = \mathcal{N}(m_0,C)$, with mean $m_0 \in \mathbb{R}^{d_M}$ and covariance $C \in \mathbb{R}^{d_M \times d_M}$, the optimal rank $r$ basis $V_r \in \mathbb{R}^{d_M \times r}$ for representing samples $m_i \sim \nu$ is given by the $r$ eigenvectors corresponding to the $r$ biggest eigenvalues of the covariance matrix $C = \Psi \text{diag}(c) \Psi^T$, with $(c_i,\Psi_i)_{i \geq 1}$ representing the eigenpairs of $C$ ordered such that $c_1 \geq c_2 \geq \dots \geq c_{d_M} \geq 0$, i.e
\begin{equation}
	\min_{V_r \in \mathbb{R}^{d_M \times r}} \mathbb{E}_{m_i \sim \nu}[\|(I_d - V_rV_r^T)(m_i - m_0)\|^2_2] = \sum_{i = r+1}^{d_M} c_i^2.
\end{equation}
See \cite{ZahmConstantinePrieurEtAl2020}.

The input data are restricted to the dominant modes of the covariance of the parameter distribution $\nu$. This approach is appropriate in regimes where the output variability is dominated by the parameter variability in the leading input modes, and the parameter covariance has rapidly decaying eigenvalues.

Input dimension reduction using KLE has limitations since it does not explicitly take into account the mapping $m \mapsto q$, it is only dependent on the distribution $\nu$. For this reason we propose to use derivative-informed input dimension reduction based on the map Jacobian information. The Jacobian of the outputs with respect to the parameters can be used to find a global subspace in which the outputs are most sensitive to the input parameters over the parameter space. Similar techniques have been popularized for dimension reduction for scalar valued functions under the name ``active subspaces'' (AS) \cite{ConstantineDowWang2014}, for which projection error bounds can be established using Poincar\'{e} inequalities. The ideas have been generalized to vector valued functions \cite{ZahmConstantinePrieurEtAl2020}, or scenarios where Poincar\'{e} inequalities do not hold \cite{ParenteWallinWohlmuth2020}.

The active subspace is constructed by the eigenvectors of the Jacobian information matrix, or a Gauss-Newton Hessian of $\frac{1}{2}\|q(m)\|^2_{\ell^2(\mathbb{R}^{d_Q})}$ averaged with respect to the parameter distribution, i.e.,
\begin{equation} \label{averaged_gn_hess}
  H_\text{GN} = \mathbb{E}_\nu [\nabla q ^T \nabla q ] = \int_M \nabla q(m)^T\nabla q(m) d\nu(m) \in \mathbb{R}^{d_M\times d_M}.
\end{equation}
The leading eigenvectors corresponding to the largest eigenvalues of $H_\text{GN}$ represents directions in which the parameter-to-output map $m \mapsto q$ is most sensitive to the input parameters $m \in \mathbb{R}^{d_M}$ in the $\ell^2(\mathbb{R}^{d_Q})$ sense (i.e. as informs a regression problem). 

Consider the ridge function parametrizations given by the $g(P_rm)$, where $P_r \in \mathbb{R}^{d_M \times d_M}$ is a rank-r projector. In \cite{ZahmConstantinePrieurEtAl2020} upper bounds are established for the approximation error of $m \mapsto q$ using a conditional expectation of $q$ with restricted to the dominant modes of the averaged Gauss-Newton Hessian (Equation \ref{averaged_gn_hess}). Specifically let $(\lambda,v_i)_{i\geq 1}$ denote the eigenpairs of the generalized eigenvalue problem
\begin{equation}\label{generalized_EVP_AS}
  H_\text{GN}v_i = \lambda_i C^{-1}v_i,
\end{equation}
then an upper bound for the approximation error $\|q - \mathbb{E}_\nu[q|\sigma(P_r)]\|^2_{\mathcal{H}}$ can be obtained (see Proposition 2.6 in \cite{ZahmConstantinePrieurEtAl2020}):
\begin{equation}\label{trailing_bound_active_subspace}
  \|q - \mathbb{E}_\nu[q|\sigma(P_r)]\|_\mathcal{H}^2 \leq \sum_{i=r+1}^{d_M} \lambda^M_i,
\end{equation}
when the projector is taken to be 
\begin{equation} \label{as_projector_definition}
  P_r = V_rV_r^TC^{-1}.
\end{equation}
Here the function space $\mathcal{H}  = L^2(\mathbb{R}^{d_M},\nu;\mathbb{R}^{d_Q}) $, and the $\mathcal{H}$ norm is induced by the inner product:
\begin{equation}
  (u,v)_\mathcal{H} = \int_{\mathbb{R}^{d_M}} (u(m),v(m))_{\ell^2(\mathbb{R}^{d_Q})} d\nu(m).
\end{equation}
The conditional expectation $\mathbb{E}_\nu[q|\sigma(P_r)](m)$ represents the mapping $m\mapsto q$ with orthogonal complement to $\text{span}\{P_r\}$ in the parameter space marginalized out. In the case that the projector $P_r$ is orthogonal with respect to the covariance $C$, the conditional expectation with respect to the $\sigma$-algebra generated by $P_r$ can be written as 
\begin{equation} \label{q_conditional_expectation}
  \mathbb{E}_\nu[q|\sigma(P_r)](m) = \mathbb{E}_{y \sim \nu}[q(P_rm + (I_{d_M} - P_r)y)].
\end{equation}

This bound establishes that when the spectrum of the Gauss-Newton Hessian decays rapidly, the mapping $m \mapsto q$ can be well approximated in expected value by the conditional expectation $\mathbb{E}_\nu[q|\sigma(P_r)](m)$, which is marginalized to the subspace spanned by the dominant modes of $H_\text{GN}$.

When the mapping $q$ is Lipschitz continuous with constant $L \geq 0$, and the eigenvalues of the parameter distribution covariance decay quickly, the KLE can be used to construct a similar ridge function bound, i.e. there exist a function $g$ and projector $P_r$ such that

\begin{equation}\label{eq:KLE-error}
  \| q - g \circ P_r\|_\mathcal{H}^2 \leq L^2 \sum_{i=r+1}^{d_M} c_i,
\end{equation}
where $c_i$ is the $i^{th}$ eigenvalue of the covariance $C$. Moreover, it is established (by Proposition 3.1 in \cite{ZahmConstantinePrieurEtAl2020}) that the upper bound of the active subspace projection error in \eqref{trailing_bound_active_subspace} is smaller or equal to the upper bound of the KLE projection error in \eqref{eq:KLE-error}, i.e.,
\begin{equation}\label{trailing_bound_active_subspace_vs_kle}
\sum_{i=r+1}^{d_M} \lambda^M_i \leq L^2 \sum_{i=r+1}^{d_M} c_i.
\end{equation}

Both KLE and AS can be used to detect low dimensionality of the input space, as well as to provide dominant basis vectors for the input space. KLE can be used to reduce the input dimensionality when the covariance of the parameter distribution $\nu$ has quick decay. The issue with KLE is that it does not take into account the outputs $q$. In contrast, AS can be used to reduce the input dimensionality when the covariance preconditioned Gauss-Newton Hessian has quick eigenvalue decay, which directly takes into account the sensitivity of the outputs to the input parameters and can be used in a more broad set of circumstances than KLE. It is likely that the AS projection is more accurate than the KLE projection, informed by the relation of their error bounds \eqref{trailing_bound_active_subspace_vs_kle}, and demonstrated by our numerical experiments.

The limitations of the AS approach are that the parametric mapping is
assumed to be differentiable with respect to the parameters $m$, and
that the computation of the AS basis is formally more expensive than
the computation of the KLE basis. Both bases can be calculated via
matrix free randomized methods \cite{HalkoMartinssonTropp2011}, which do not require the explicit
formation of the matrices, and instead require only matrix-vector
products. KLE requires access to just the covariance operator, while
AS require evaluations of the Jacobian of the parametric mapping at
different sample points. However, the Jacobian matrix-vector products
required in the computation of the AS subspace require only marginally
more work than the computation of the training data. This is because
once the forward PDE is solved at a sample point to construct a
training data pair, the Jacobian matrix-vector product at that point
requires solution of linearized forward and adjoint PDE problems with
$O(r)$ right hand sides each.  When direct (or iterative solvers with
expensive-to-construct preconditioners) are employed to solve these
problems, little additional work is required beyond the forward
solve. This is because triangular solves are asymptotically cheap
compared to the factorizations (or for iterative solves, the
preconditioner needs to be constructed just once).  Moreover, blocking
the $O(r)$ right hand sides in the forward/adjoint solves results in
better cache performance, thereby making the Jacobian action less
expensive that it might seem. 
In practice AS may require fewer samples than training data
construction (in numerical experiments we use significantly fewer
samples for AS construction than training data). When the spectrum of
the parameter covariance-preconditioned Gauss-Newton Hessian decays
rapidly, $r$ is small and only few matrix-vector products (and hence
linearized forward/adjoint solves) are required. For more information
see Appendix \ref{jacobian_with_adjoints}.

\subsection{Output Projection: Proper Orthogonal Decomposition}

Reduced order modeling (ROM), e.g. a reduced basis method (RBM) have been developed to help reduce the dimension of PDE outputs (solution and quantities of interest) \cite{BennerGugercinWillcox2015,ChenQuarteroniRozza2013,ChenQuarteroniRozza2017,QuarteroniManzoniNegri2015,WillcoxPeraire2002}. This methodological framework has made tractable the solution of many-query problems involving PDEs (optimization, inference, control, inverse problems etc.) \cite{BuiDamodaranWillcox2004}. In these methods low dimensional representations of the outputs are made via snapshots.
Due to the optimality of the representation in the sense of Proposition \ref{POD_bound_proposition}, this approach is also considered for solving parametric PDEs using dimension reduced neural networks \cite{BhattacharyaHosseiniKovachki2020}.

The proper orthogonal decomposition (POD) for the outputs is the eigenvalue decomposition of the expectation of the following matrix,
\begin{equation}
  \mathbb{E}_\nu[qq^T] = \int_{\mathbb{R}^{d_M}} q(m)q(m)^T d\nu(m) = \Phi D\Phi^T.
\end{equation}

Using the Hilbert-Schmidt Theorem, the POD basis is shown to be optimal for the following minimization problem (see \cite{ManzoniNegriQuarteroni2016,QuarteroniManzoniNegri2015}).

\begin{proposition}[Proposition 2.1 in \cite{ManzoniNegriQuarteroni2016}] \label{POD_bound_proposition}
The POD basis $\Phi \in \mathbb{R}^{d_Q \times r_Q}$ is such that
\begin{align}
  &\int_{\mathbb{R}^{d_M}}\| q(m) - \Phi\Phi^Tq(m)\|^2_{\ell^2(\mathbb{R}^{d_Q})} d\nu(m) \nonumber \\
   = \min_{W \in \mathbb{R}^{d_Q \times r_Q}, W^TW = I_{r_Q}} &\int_{\mathbb{R}^{d_M}}\| q(m) - WW^Tq(m)\|^2_{\ell^2(\mathbb{R}^{d_Q})} d\nu(m).
\end{align}
And the error is given by the trailing eigenvalues of $\mathbb{E}_\nu[qq^T]$:
\begin{equation}
  \int_{\mathbb{R}^{d_M}}\| q(m) - \Phi\Phi^Tq(m)\|^2_{\ell^2(\mathbb{R}^{d_Q})} d\nu(m) = \sum_{i = r_Q +1}^{\text{rank}(T)} \lambda^Q_i
\end{equation}
\end{proposition}

POD serves as a constructive prescription for a low rank basis that is optimal in the $\ell^2$ sense. It also serves as a means of reliably calculating the inherent dimensionality of the outputs for the  mapping $m \mapsto q$. The POD basis can be approximated via Monte Carlo samples directly from training data using the ``method of snapshots'' \cite{Pinnau08}.

\subsection{Input--Output Error Bound for Optimal Ridge Function} \label{section_input_output_bound}

Reduced bases for the parameter space $\mathbb{R}^{d_M}$ and the
output space $\mathbb{R}^{d_Q}$ can be used to design input--output
reduced ridge function surrogates for the mapping $m \mapsto q$ as in
Equation \eqref{input_output_ridge_function}. When the input
dimensions that inform the outputs are well approximated by the modes
spanned by $V_{r_M}$, and the outputs are well approximated when
restricted to the span of $\Phi_{r_Q}$, then the mapping $m \mapsto q$
can be well approximated by ridge functions of this form. A schematic
for this ridge function strategy is given in Figure
\ref{ridge_schematic}. 

\begin{figure}[H]
\center
\begin{tikzpicture}[scale = 0.9,every node/.style={draw,outer sep=0pt,thick}]
\node[bag] (Input) at (0,0) [minimum width=1cm,minimum height=3cm] {Input basis\\ $\mathbb{R}^{d_M} \rightarrow \mathbb{R}^{r_M}$};
\node[bag] (NN) at (4,0) [minimum width=3cm,minimum height=0.6cm] {$f_r:\mathbb{R}^{r_M}\rightarrow \mathbb{R}^{r_Q}$};
\draw (Input.north east) -- (NN.north west) (Input.south east) -- (NN.south west);
\node[bag] (Output) at (8,0)[minimum width=1cm,minimum height=1.2cm] {Output basis\\ $\mathbb{R}^{r_Q} \rightarrow \mathbb{R}^{d_Q}$};
\draw (NN.north east) -- (Output.north west) (NN.south east) -- (Output.south west);
\end{tikzpicture}
\caption{Dimension reduced representation by conditional expectation ridge function.}
\label{ridge_schematic}
\end{figure}
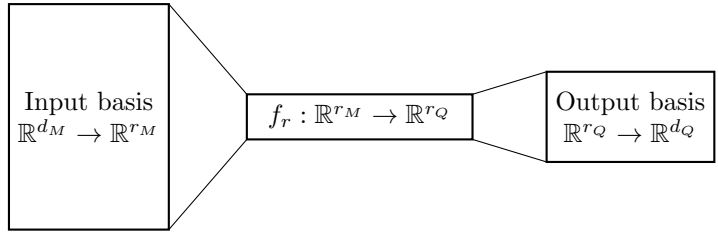

Combining each of the KLE and AS approaches with POD, an error bound can be established for the conditional expectation ridge function that is restricted to the dominant modes of the POD basis.

\begin{proposition}[Input--Output Ridge Function Error Bound]\label{prop_input_output_error_bound}

Let $P_{r_M} = V_rV_r^TC^{-1}\in\mathbb{R}^{d_M}$ be the projectors coming from the active subspace generalized eigenvalue problem \eqref{generalized_EVP_AS}, and define the conditional expectation of the outputs $q$ with respect to the $\sigma$-algebra generated $P_{r_M}$ given in Equation \eqref{as_projector_definition}:
\begin{equation}
  q_{r_M}(m) = \mathbb{E}_\nu[q|\sigma(P_{r_M})](m).
\end{equation}
Define the rank $r_Q$ POD decomposition for $q_{r_M}$ as follows:
\begin{equation}
  \big[\mathbb{E}_\nu[q_{r_M}q_{r_M}^T]\big]_{r_Q} = \bigg[\int_{\mathbb{R}^{d_M}}q_{r_M}(m)q_{r_M}(m)^T d\nu(m)\bigg]_{r_Q} = \widehat{\Phi}_{r_Q}\widehat{D}_{r_Q} \widehat{\Phi}^T_{r_Q}.
\end{equation}
Then one has
\begin{equation}
  \int_{\mathbb{R}^{d_M}} \|q(m) - \widehat{\Phi}_{r_Q}\widehat{\Phi}_{r_Q}^Tq_{r_M}(m)\|^2_{\ell^2(\mathbb{R}^{d_Q})}d\nu(m) \leq \sum_{i=r_M+1}^{d_M}\lambda_i^M + \sum_{i=r_Q +1}^{d_Q} \lambda_i^Q 
\end{equation}
Further if $q$ is Lipschitz continuous with constant $L\geq 0$, then
\begin{equation}
  \int_{\mathbb{R}^{d_M}} \|q(m) - \widehat{\Phi}_{r_Q}\widehat{\Phi}_{r_Q}^Tq_{r_M}(m)\|^2_{\ell^2(\mathbb{R}^{d_Q})}d\nu(m) \leq L^2\sum_{i=r_M+1}^{d_M}c_i + \sum_{i=r_Q +1}^{d_Q} \lambda_i^Q 
\end{equation}
\end{proposition}

\begin{proof}
This result follows from the triangle inequality
\begin{align}
  &\int_{\mathbb{R}^{d_M}} \|q(m) - \widehat{\Phi}_{r_Q}\widehat{\Phi}_{r_Q}^Tq_{r_M}(m)\|^2_{\ell^2(\mathbb{R}^{d_Q})}d\nu(m)  \nonumber\\ 
  \leq &\int_{\mathbb{R}^{d_M}} \|q(m) - q_{r_M}(m)\|^2_{\ell^2(\mathbb{R}^{d_Q})} + \|q_{r_M}(m) - \widehat{\Phi}_{r_Q}\widehat{\Phi}_{r_Q}^Tq_{r_M}(m)\|^2_{\ell^2(\mathbb{R}^{d_Q})}d\nu(m),
\end{align}
and application of \eqref{trailing_bound_active_subspace},\eqref{trailing_bound_active_subspace_vs_kle} and Proposition \ref{POD_bound_proposition}.
\end{proof}

This result establishes that when the spectra for $\mathbb{E}_\nu[\nabla q^T \nabla q]$ and $\mathbb{E}_\nu[qq^T]$ decay rapidly, the map $m\mapsto q$ can be well approximated over $\nu$ when restricted to the dominant subspaces of AS and POD. Further if $q$ is Lipschitz continuous with constant $L$ not too large, and the covariance of $\nu$ has quick eigenvalue decay, then the map can be well approximated when restricted to the dominant subspace of KLE and POD.

If the mapping $m \mapsto q$ satisfy these conditions, then we just need to find a low dimensional ridge function mapping $f_r:\mathbb{R}^{r_M} \rightarrow \mathbb{R}^{r_Q}$. In the next section we consider feedforward neural networks for the nonlinear ridge functions. This input--output dimension reduction strategy is general and can be readily extended to different nonlinear representations such as polynomials, Gaussian process etc.

\section{Projected Neural Network} \label{section_nn_architecture}

In this section we present strategies for constructing projecting neural networks. A projected neural network, parametrized by weights $\mathbf{w} = [w,b_Q] \in \mathbb{R}^{d_W}$, can be written as follows:
\begin{equation}
  f(m,\mathbf{w}) = \Phi_{r_Q} f_r(V_{r_M}^Tm,w) + b_Q.
\end{equation}

The function $f_r$ represents a sequence of nonlinear compositions of affine mappings between successive latent representation spaces $\mathbb{R}^{d_\text{layer}}$. The architecture of the neural network includes the choice of layers, layer dimensions (widths) and nonlinear activation functions; for more information see \cite{GoodfellowBengioCourville2016}. Once an architecture is chosen, optimal weights are found via the solution of an empirical risk minimization problem over given training data. Here we consider least squares regression: given $\{(m_i,q_i)|m_i \sim \nu\}_{i=1}^{N_\text{train}}$, find the weights $\mathbf{w}$ that minimize
\begin{equation} \label{nn_least_squares}
	\min_{\mathbf{w} \in \mathbb{R}^{d_W}}\frac{1}{N_\text{train}}\sum_{i=1}^{N_\text{train}} \|f(m_i,\mathbf{w}) - q(m_i)\|^2_{\ell^2(\mathbb{R}^{d_Q})}.
\end{equation}

In the setting we consider, a limited number of data are expected to be available for the empirical risk minimization problem. A general rule of thumb is to have a number independent training data $N_\text{train}$ that is commensurate to the number of weights, $d_W$, to be inferred, (for linear problems $N_\text{train} > d_W / d_Q$ linearly independent data are needed). Since we cannot afford many training data, instead we try to keep $d_W$ small, and the network parsimonious.

A common dimension reduction strategy for regression problems in machine learning is to use an encoder-decoder network \cite{Kramer1991}. In an encoder-decoder input data is nonlinearly contracted to lower dimensions and then extended to the outputs. The main issue in designing an encoder-decoder is to find a reduced dimensionality that is appropriate. In this setting the intrinsic dimensionalities of the inputs and outputs are exposed by AS, KLE and POD. These decompositions can be used to help identify appropriate dimensions for encoder-decoder architectures. 

Additionally the basis representations that result from these decompositions can be used explicitly in the neural network configuration. When training a full neural network (i.e. including the first and last layers as weights) the projectors can be used as initial guesses for neural network training, which can ease the difficulty of the optimization problem, by initializing the portions of the weights corresponding to the first and last layers to regions of the parameter space and outputs that are informed by the parameters $m$ or the parameter-to-output map. 

Otherwise one can architect a parsimonious network by keeping the
input or output projectors fixed. In this case the dimension of the
trainable neural network does not depend explicitly on $d_M$ or $d_Q$,
but only on $r_M$ and $r_Q$. The network is constrained to only learn
patterns between the dominant subspaces of the parameter space and
outputs, as exposed by the projectors. If the dependence of the weight
dimension on $d_M$ or $d_Q$ is removed by constraining the neural
network to the dominant subspaces of the parameter space and outputs,
the weights for the neural network can be inferred using few training
data. This makes this strategy attractive in settings where the input
or output dimensions are very large and the mapping needs to be
queried many times, as in many many-query applications including inverse
problems, optimal experimental design, forward uncertainty
quantification, and optimal control and design.

\subsection{Discussion of Errors} \label{section:discussion_of_errors}

In Section \ref{section_input_output_bound} we discussed errors in approximating $m \mapsto q$ by a ridge function restricted to the dominant modes of both the inputs and outputs. The bounds there are for projection errors. When we attempt to represent the map by a neural network ridge function other errors are introduced. There is an error in the approximation of the low dimensional mapping by the neural network instead of other nonlinear ridge functions. We can think of this error as
\begin{equation}
    \int_{\mathbb{R}^{d_M}}\|f(m,\mathbf{w}) - q_{r_M}(m)\|_{\ell^2(\mathbb{R}^{d_Q})}^2 d\nu(m) = \text{Representation Error}.
\end{equation}
Some theoretical guarantees have been established for neural network
representation errors; see for example
\cite{SchwabZech19,MaWuE2019,MaWojtowytschWuEtAl20} 
As was discussed in the previous section, the neural network weights are found via the solution of an empirical risk minimization problem, for which we use Monte Carlo samples to approximate the integral with respect to the parameter distribution $\nu$, which incurs a sampling error that leads to a generalization gap \cite{MaWojtowytschWuEtAl20}. The empirical risk minimization problem is nonconvex, and so finding a global minimizer is NP-hard. One has to instead settle for local minimizers that perform reasonably well, and this introduces an additional optimization error:
\begin{subequations}
\begin{equation}
  \mathbf{w}^\dagger \neq \mathbf{w}^* = \text{argmin} \frac{1}{N_\text{train}}\sum_{i=1}^{N_\text{train}} \| f(m_i,\mathbf{w})- q(m_i)\|_{\ell^2(\mathbb{R}^{d_Q})}^2. 
\end{equation}
The approximate solution ($\mathbf{w}^\dagger$) to the nonconvex optimization problem, is generally different than the global minimizer $(\mathbf{w}^*)$, and the difference between the local minimum and the global minimum is the optimization error
\begin{align}
  \bigg|\frac{1}{N_\text{train}}\sum_{i=1}^{N_\text{train}} \| f(m_i,\mathbf{w}^\dagger) - q(m_i)\|_{\ell^2(\mathbb{R}^{d_Q})}^2 - \| f(m_i,\mathbf{w}^*) - q(m_i)\|_{\ell^2(\mathbb{R}^{d_Q})}^2\bigg| \nonumber \\
  = \text{Optimization Error}.
\end{align}
\end{subequations}
Recent work suggests that for feedforward neural networks this error can be mitigated by an iterative layerwise optimization procedure \cite{ChanYuYouEtAl20}. An additional sampling error is incurred in the Monte Carlo approximations made for the computation of the AS and POD projectors. Bounds for such errors are investigated in \cite{HolodnakIpsenSmith18}. In total, the errors incurred can be decomposed into projection error, representation error, sampling errors and optimization error. Each of these errors plays a role in the final solution one gets when constructing a surrogate of this form.

\section{Numerical Experiments}

We conduct numerical experiments on input--output projection based neural networks on two PDE based regression problems. We consider derivative informed projected neural networks (DIPNets) which use AS for input projection and POD for output projection, we compare this strategy against KLE for input projection and POD for the output projection. Proposition \ref{prop_input_output_error_bound} suggests that the decay of the eigenvalues for KLE, AS and POD can be used to choose the ranks for these architectures. In order to have a consistent error tolerance criterion between AS and KLE, the Lipschitz constant $L$ appearing in the upper bound \eqref{trailing_bound_active_subspace_vs_kle} must be used for comparison. Since this constant is not known a priori, we use fixed ranks in order to have a fair comparison between DIPNet and KLE-to-POD. For simplicity we take the input and output ranks to be the same to demonstrate the generalization accuracy of the projected neural networks compared to the full space network. More representative projected architectures could be achieved by more detailed rank selection procedures, but at additional computational cost.

To demonstrate the benefit of using the projection basis vectors to construct the neural networks, we also use Gaussian random vectors for comparison. Moreover, to demonstrate the projection errors separately, without involving the neural network representation and optimization errors, we also consider the following projection error

\begin{equation}\label{input_output_projection_error}
  \mathbb{E}_{m \sim \nu}\left[\frac{\|P_\text{POD}q(P_\text{input}m) - q(m)\|_{\ell^2(\mathbb{R}^{d_Q})}}{\|q(m)\|_{\ell^2(\mathbb{R}^{d_Q})}}\right].
\end{equation}
Here $P_\text{POD}$ is the projector on to the principal subspace for the outputs as represented by POD, and $P_\text{input}$ is the input projector (either AS or KLE). These errors are studied as a function of rank.

For the training of the projected neural network, once the architectures (including projection bases, network layers, layer dimensions, and nonlinear activation functions) are chosen, the locally optimal weights $\mathbf{w}^\dagger \in \mathbb{R}^{d_W}$ are found via solving an empirical risk minimization problem. Given training data $X_\text{train} = \{(m_i,q(m_i))\}_{i=1}^{N_\text{train}}$, candidate optimal weights are found via solving the empirical risk least-squares minimization problem,
\begin{equation} \label{model_arch_emp_risk}
  \min_{w\in \mathbb{R}^{d_W}} F_{X_\text{test}}(w) = \frac{1}{2N_\text{train}} \sum_{i=1}^{N_\text{train}} \| q(m_i) - f(m_i,\mathbf{w})\|^2_{\ell^2(\mathbb{R}^{d_Q})}.
\end{equation}
We used a subsampled inexact Newton CG method because it performed reliably better than Adam, SGD, and other first order methods on the problems we tested \cite{OLearyRoseberryAlgerGhattas2019}. We define the relative error as
\begin{equation}
  \text{Relative Error} =  \sqrt{\frac{\sum_{i=1}^{N_\text{test}}\|q(m_i) - f(m_i,\mathbf{w}^\dagger)\|^2}{\sum_{i=1}^{N_\text{test}}\|q(m_i)\|^2}},
\end{equation}
and the accuracy as 
\begin{equation}
  \text{Accuracy} = 100(1 - \text{Relative Error}).
\end{equation}

We are particularly interested in what surrogates can achieve good performance with only a limited number of training data, as these are the conditions of the problems that motivate this work. We study the performance of the neural networks as a function of training data, $N_\text{train}$. For each problem $2048$ total data are generated, with $512$ set aside as testing data.

Since the performance of the stochastic optimizer can be sensitive to the choice of initial guesses for the weights not given by the projectors, we average results over ten initial guesses for each neural network, and report averages and standard deviations. We are interested in neural networks that are robust to the choice of initial guess for the weights. 

We compare the projected neural networks against a full space neural network. The full space neural network (FS) maps the input data to the output dimension, $d_Q$. We consider $d_Q < d_M$, so that the FS is not too high dimensional. For both the projected and full space networks we use two nonlinear layers for simplicity and ease of training, with softplus activation functions; these choices performed well during training. For each of the projected neural networks the weights in the input layer are be fixed during training, while the weights in the output layer are trained, which we find improved the generalization accuracy. This allows for a neural network with a parsimonious training weight dimension independent of $d_M$. We do not include the effect of training the input layer in these results; these networks only performed well when the trained input reduced network weights were used as an initial guess, and even then they did not outperform the input reduced networks. This came at the expense of increasing the neural network weight dimension by orders of magnitude. The input and output projectors were orthogonalized and rescaled, which significantly improved generalization accuracy. For more details see Appendix \ref{appendix_implementation_specifics}.

The naming conventions we use for the neural networks is ``DIPNet'' for
the neural networks that use the AS projector for the input and the
POD projector for the outputs. We call the networks that use KLE for
the input and POD for the output ``KLE'' networks. We call the
Gaussian random projected networks ``RS'' for random space, and the
full space networks ``FS.'' We investigate two numerical examples
that involve parametric maps related to PDE-governed inference
problems.  

The first example is a 2D semi-linear convection-diffusion-reaction PDE, where the parameters $m$ represents a coefficient in a nonlinear reaction term. The second example is a 2D Helmholtz PDE with a single source, where the parameters $m$ represents a variable coefficient. In both examples the outputs are the PDE state evaluated at some points the interior of the PDE domain.

Both problems use a Gaussian distribution with Mat\'{e}rn covariance for the uncertain parameter distribution $\nu(m)$, for more information see Appendix \ref{section_matern_prior}. The Mat\'{e}rn covariance can be represented by the inverse of a coercive elliptic operator. The KLE basis vectors are taken as eigenvectors corresponding to the leading eigenvalues of the discrete covariance, discretized by a finite element method.

\subsection{Convection-Diffusion-Reaction Problem} \label{section_confusion}

For a first case study, we investigate a 2D quasi-linear
convection-diffusion-reaction problem with a nonlinear reaction term. 
The formulation of the problem is
\begin{subequations}
\begin{align}
  - \nabla \cdot (k \nabla u) + \mathbf{v} \cdot \nabla u + e^m u^3 &=
  f \quad \text{in } \Omega \\ 
  u &= 0 \text{ on } \partial \Omega \\
  q(m) = Bu(m) &= [u(\mathbf{x}^{(i)},m)] \quad \text{at } \mathbf{x}^{(i)} \in (0.6,0.8)^2\\
  \Omega &= (0,1)^2.
\end{align}
\end{subequations}
where  $m$ is a mean-zero random field representing the log-coefficient of the
reaction term. 
The volumetric forcing function $f$ is a Gaussian bump located at $x_1 = 0.7,x_2=0.7$. The velocity field $\mathbf{v}$ used for each simulation is the solution of a steady-state Navier-Stokes equation with side walls driving the flow. See Appendix \ref{num_exp_appendix} for more information.

Here we consider unit square meshes 
parametrized by $n_x,n_y$. We study two meshes, $n_x = n_y =
64$ and $n_x = n_y = 192$. The Gaussian Mat\'{e}rn distribution for
$\nu$ is parametrized by $\gamma = 0.1,\delta = 1.0$. We start by
investigating the eigenvalue decompositions for AS, KLE and POD, and
the input--output projection errors given by
\ref{input_output_projection_error}.

In Figure \ref{as_kle_spectra_confusion}, the AS spectra agree between the two mesh discretization and are roughly mesh independent. The KLE spectra agree between the two mesh discretization and are roughly mesh independent. The decay of the AS spectra are quicker than those of the KLE spectra.

\begin{figure}[H] 
\begin{subfigure}{0.5\textwidth}
\includegraphics[width = \textwidth]{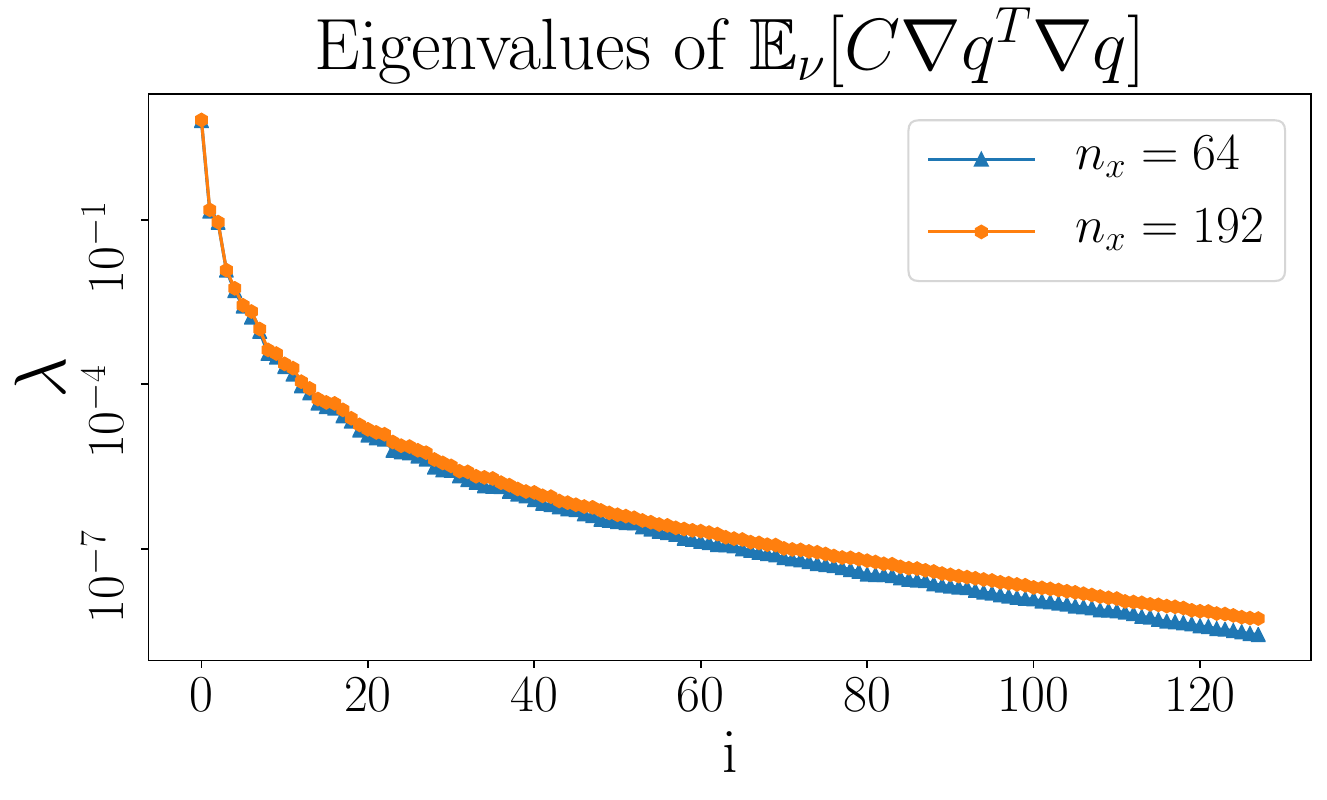}
\end{subfigure}%
\begin{subfigure}{0.5\textwidth}
\includegraphics[width = \textwidth]{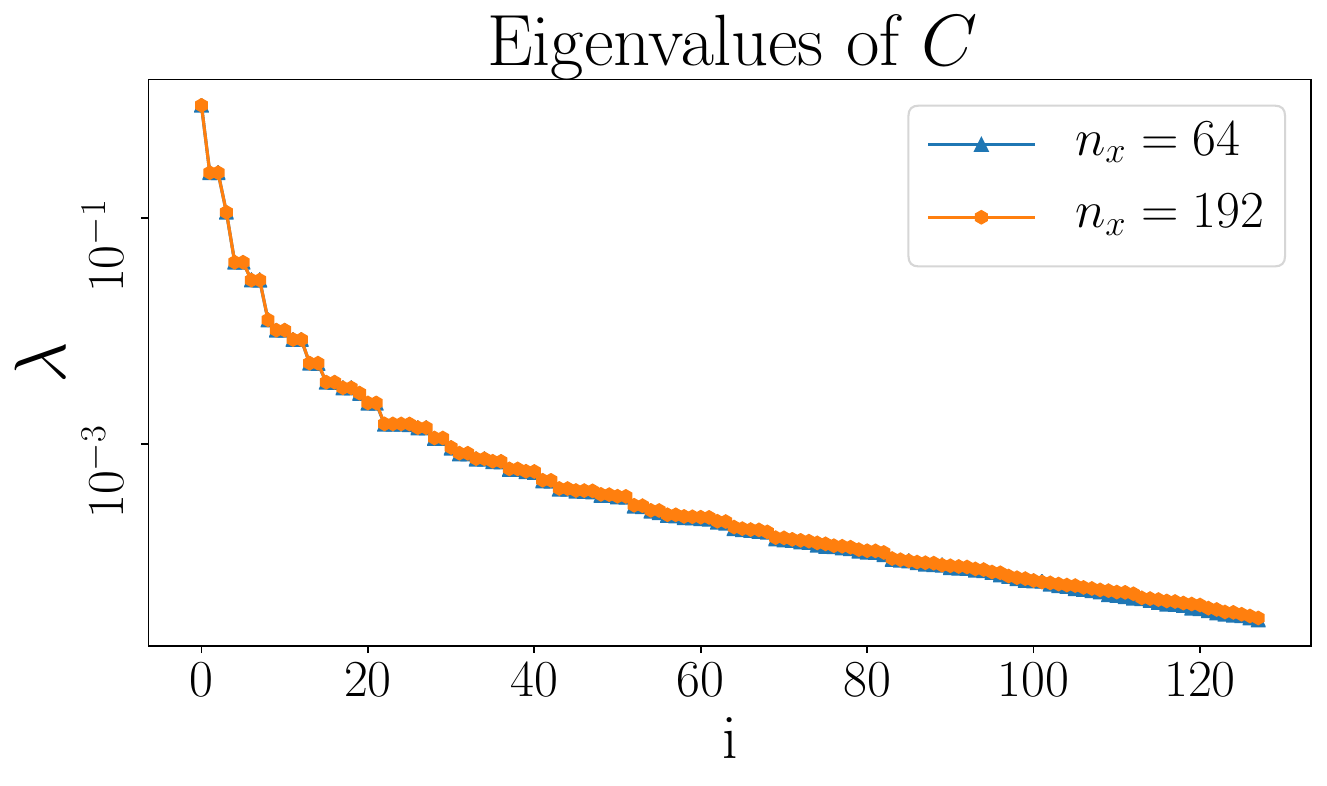}
\end{subfigure}
\caption{Active subspace and KLE eigenvalue decay for $\gamma =0.1, \delta = 1.0$}
\label{as_kle_spectra_confusion}
\end{figure}

In Figure \ref{as_kle_vectors_confusion},  the dominant vectors of AS are localized to the part of the domain where the observations are present. The KLE eigenvectors represent the dominant eigenmodes of the fractional PDE operator that is the covariance matrix. For this problem, which has Dirichlet boundary conditions and a unit square mesh, these modes correspond to the constant, and sines and cosines that arise in separation of variables. These eigenvectors do not pick up local information about the mapping $m \mapsto q$.

\begin{figure}[H] 
\begin{minipage}{0.85\textwidth}
\begin{subfigure}{0.25\textwidth}
\center
\includegraphics[width=\textwidth]{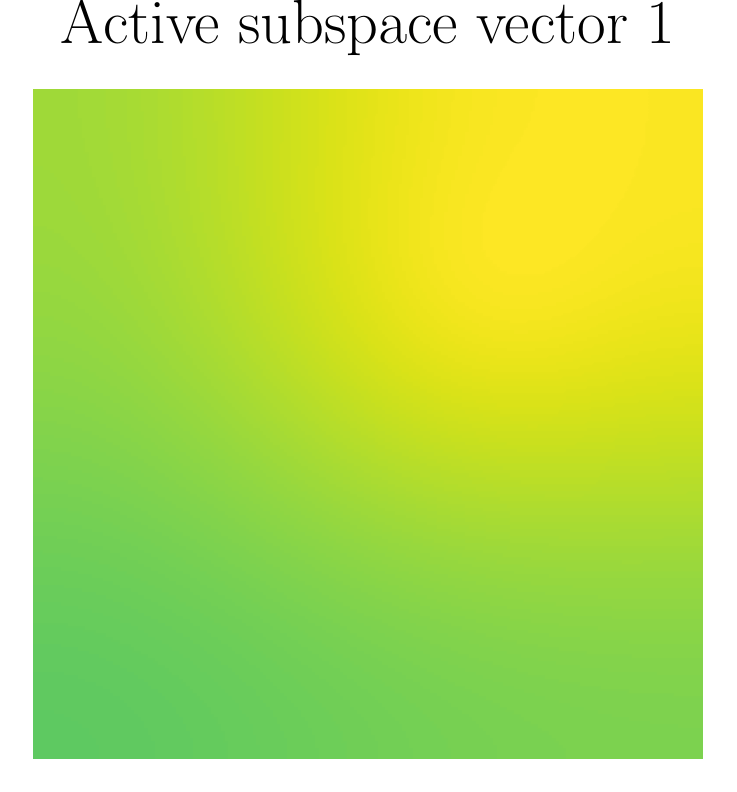}
\end{subfigure}%
\begin{subfigure}{0.25\textwidth}
\center
\includegraphics[width=\textwidth]{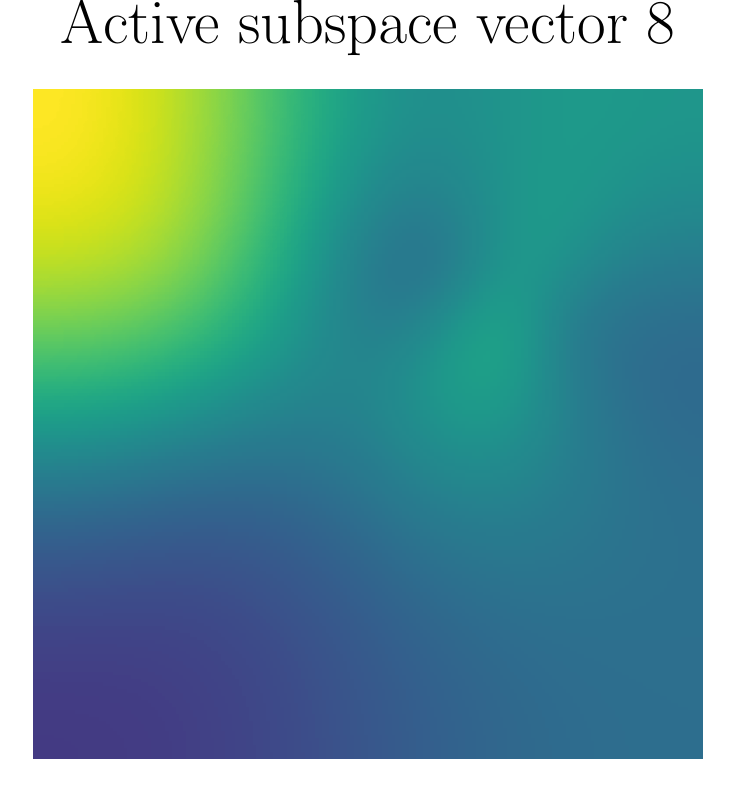}
\end{subfigure}%
\begin{subfigure}{0.25\textwidth}
\center
\includegraphics[width=\textwidth]{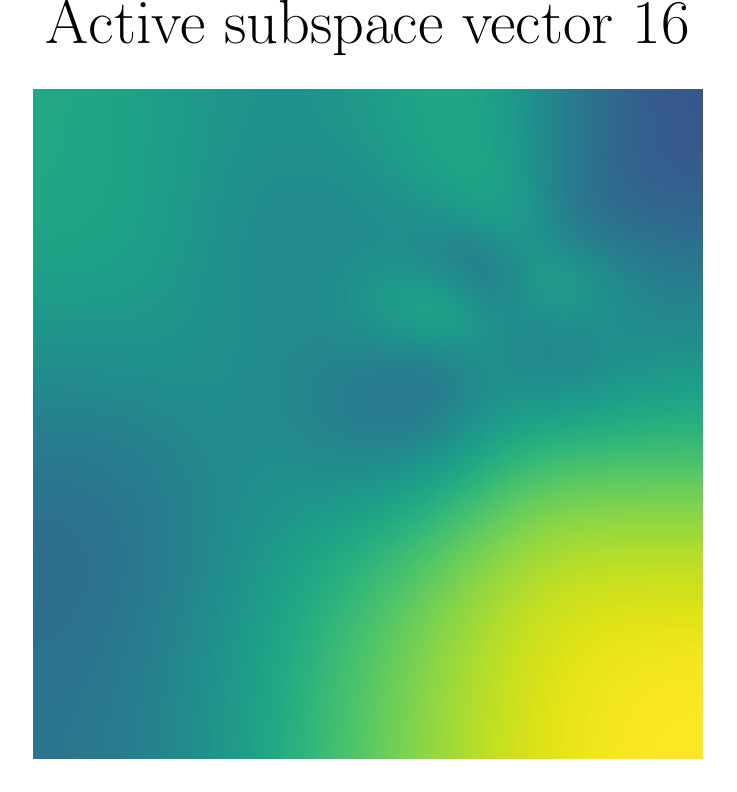}
\end{subfigure}%
\begin{subfigure}{0.25\textwidth}
\center
\includegraphics[width=\textwidth]{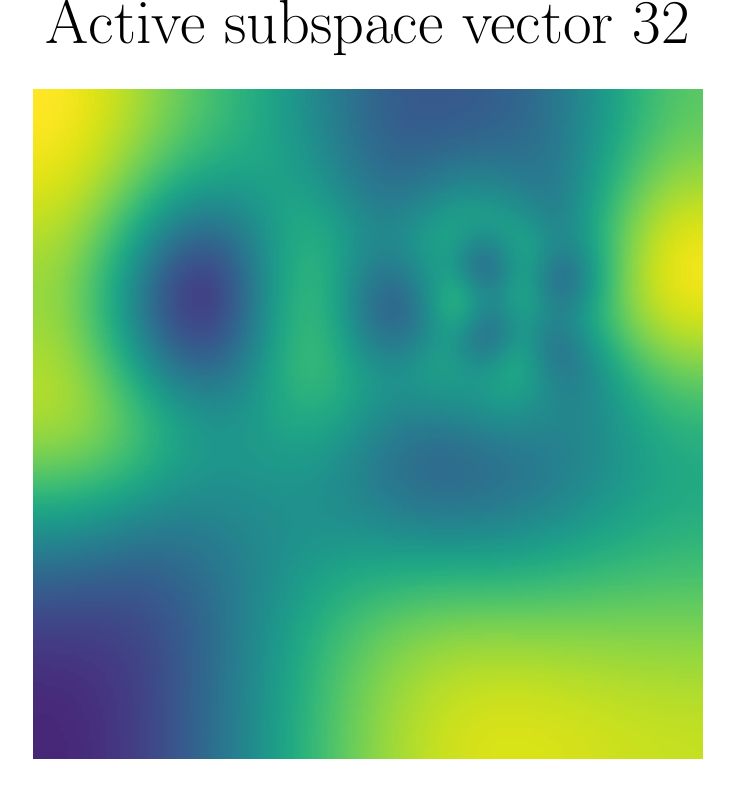}
\end{subfigure}
\begin{subfigure}{0.25\textwidth}
\center
\includegraphics[width=\textwidth]{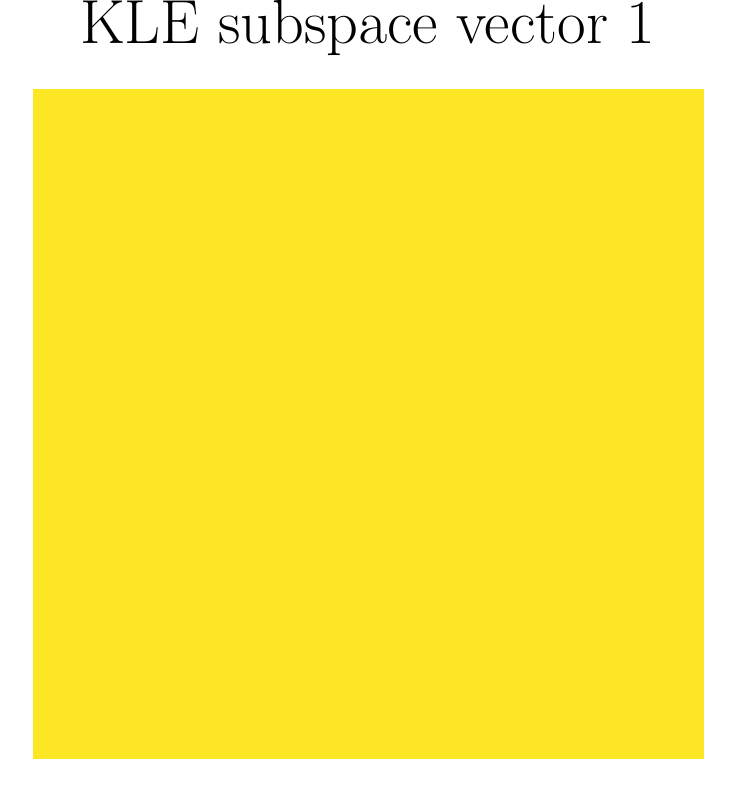}
\end{subfigure}%
\begin{subfigure}{0.25\textwidth}
\center
\includegraphics[width=\textwidth]{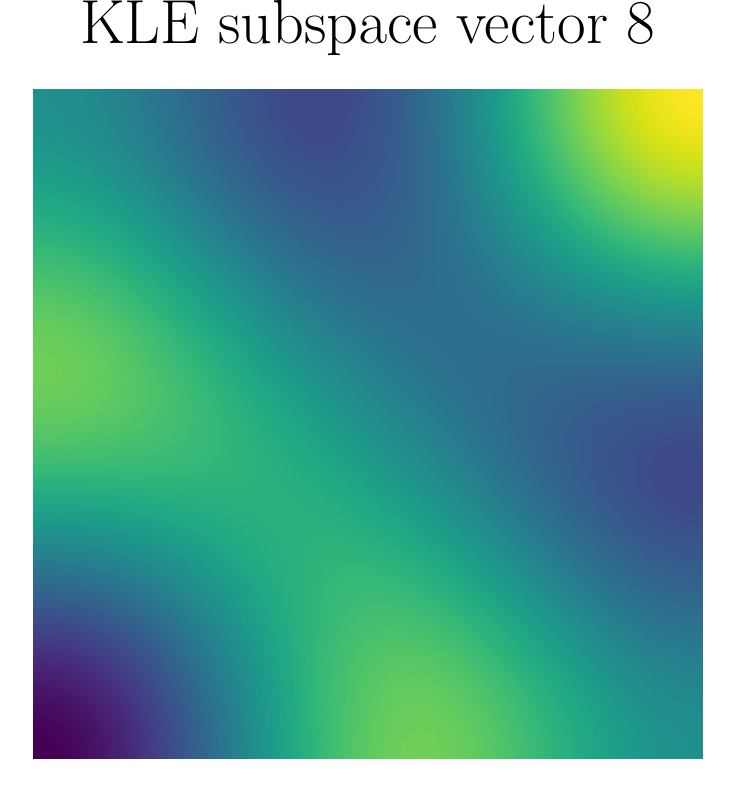}
\end{subfigure}%
\begin{subfigure}{0.25\textwidth}
\center
\includegraphics[width=\textwidth]{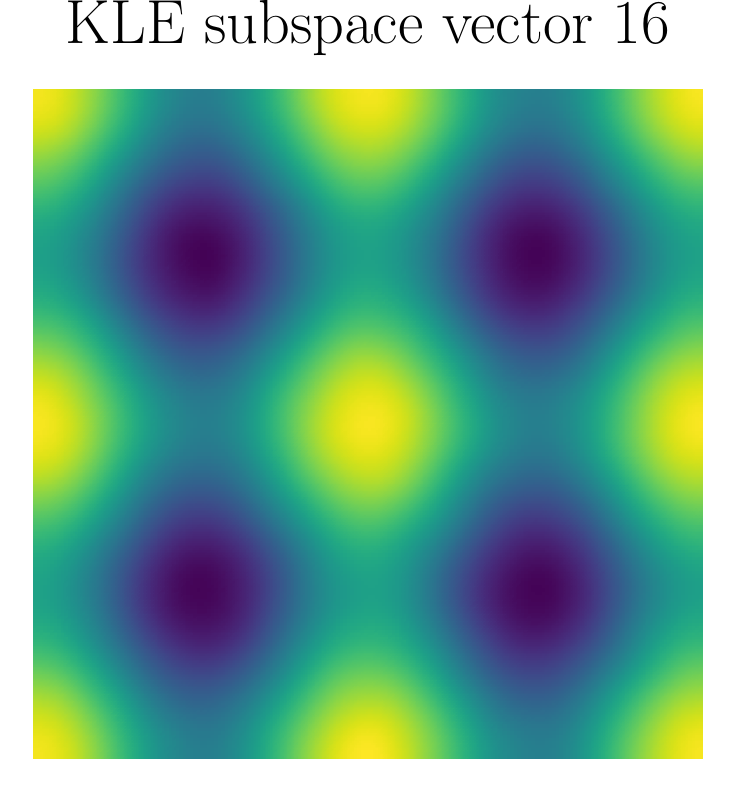}
\end{subfigure}%
\begin{subfigure}{0.25\textwidth}
\center
\includegraphics[width=\textwidth]{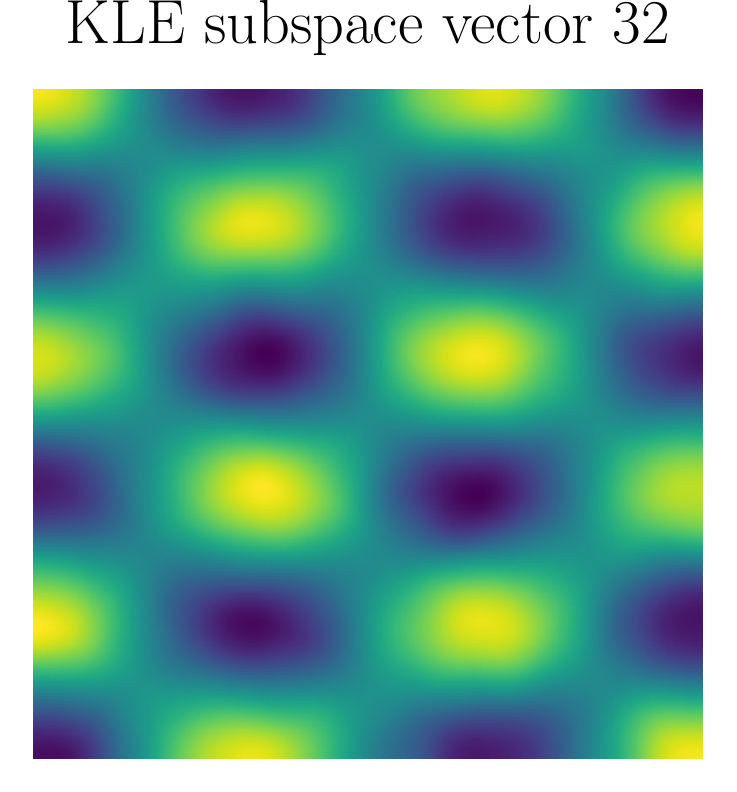}
\end{subfigure}
\end{minipage}%
\begin{minipage}{0.15\textwidth}
\includegraphics[scale = 0.6]{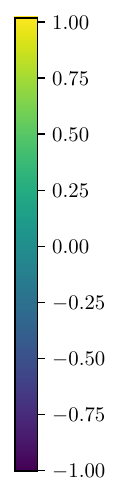}
\end{minipage}
\caption{AS and KLE eigenvectors for $\gamma =0.1, \delta = 1.0$}
\label{as_kle_vectors_confusion}
\end{figure}

In Figure \ref{pod_eigs_and_errors_confusion}, the POD spectra agree in the dominant modes and have similar qualitative decay. They however begin to diverge in the small modes, the finer discretization contains more information in the tail of the spectrum. The input--output error projection errors in Figure \ref{pod_eigs_and_errors_confusion} are sample average approximations of equation \ref{input_output_projection_error} computed for AS and KLE. Both use POD for the outputs, and in each case the ranks are taken to be the same for both the input and the output projectors. Since the output projectors are taken to be the same, this allows for a study of how informed the input projectors are. The AS projector contains more information that informs the outputs than KLE, particularly in the first $15-20$ modes, after that the decay rate of the projection errors are roughly the same. 

\begin{figure}[H] 
\begin{subfigure}{0.5\textwidth}
\includegraphics[width = \textwidth]{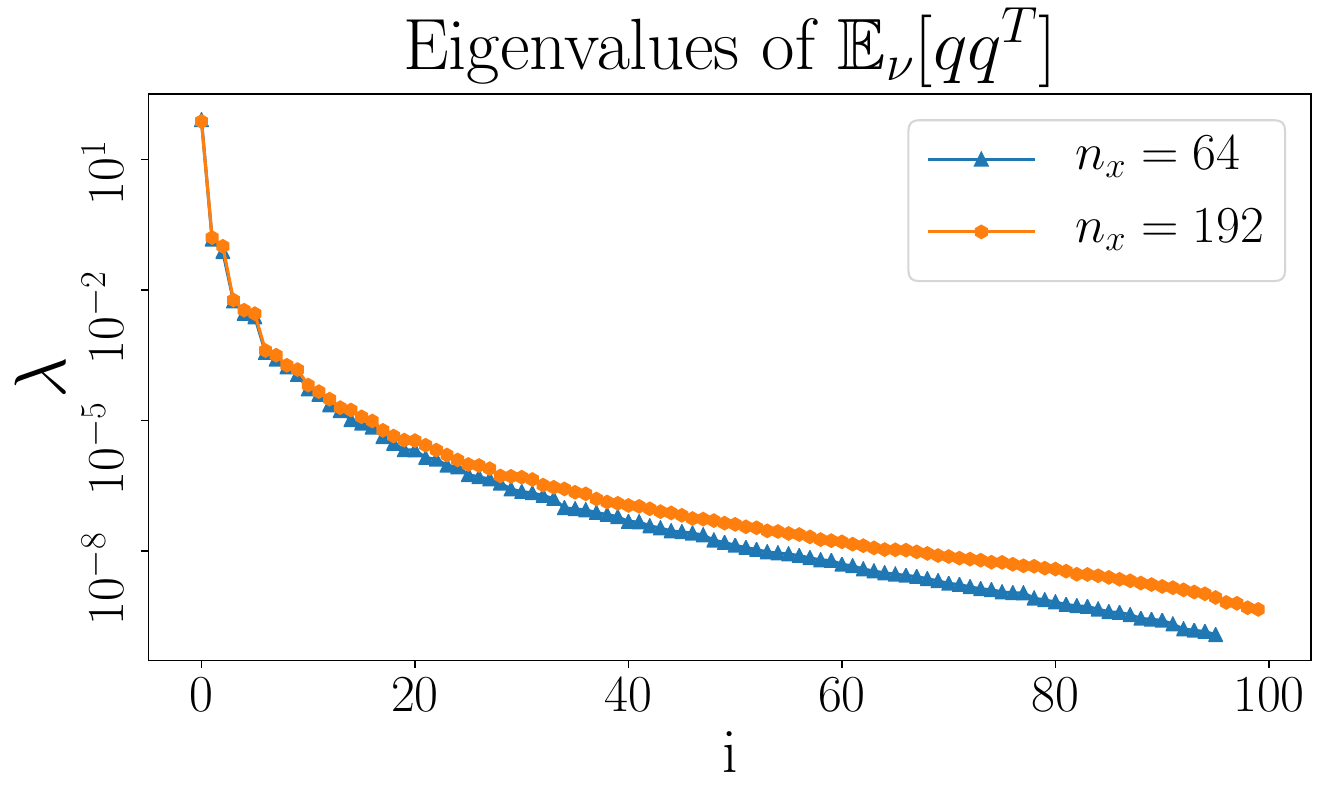}
\end{subfigure}%
\begin{subfigure}{0.5\textwidth}
\includegraphics[width = \textwidth]{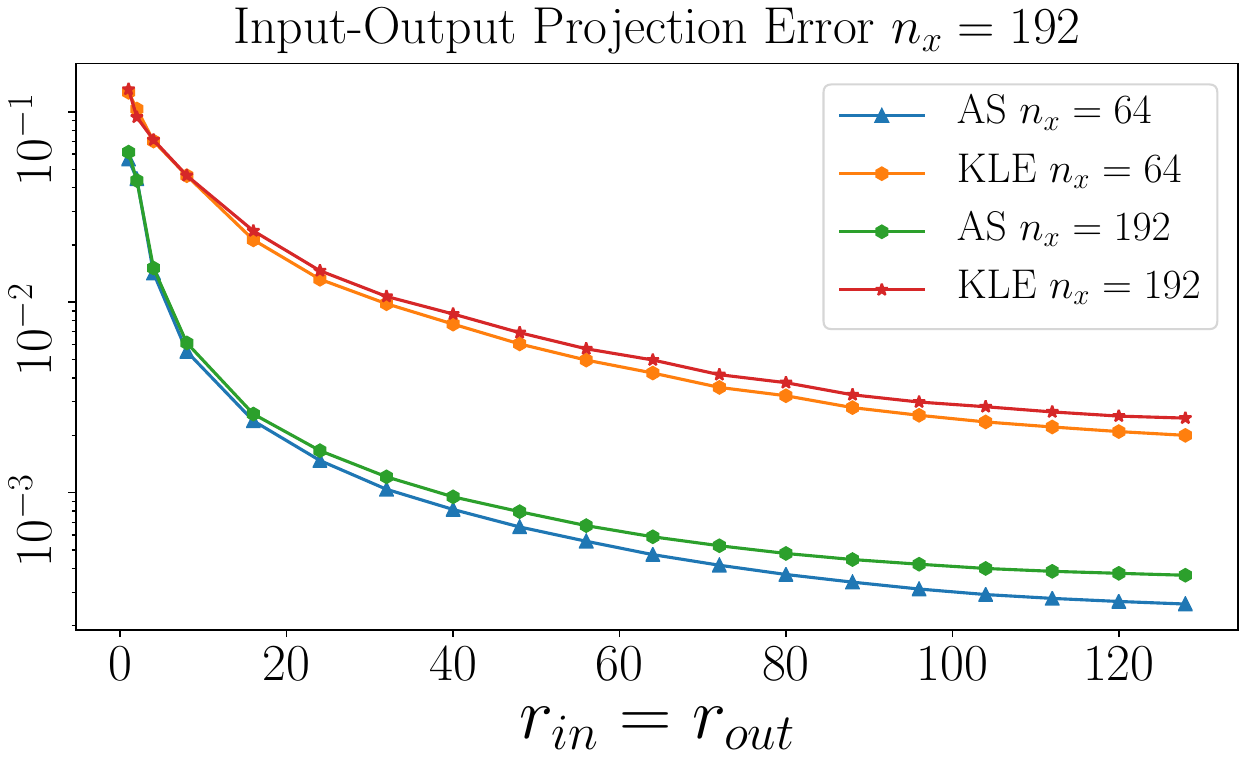}
\end{subfigure}
\caption{Plot of spectra for POD on the two meshes (left), and the input--output projection error (right) for the convection-diffusion-reaction problem}
\label{pod_eigs_and_errors_confusion}
\end{figure}

In what follows we train the projected neural networks and compare against the FS network. Below in Table \ref{table_summary_of_dimensions_confusion}, dimensions for the input parameters $d_M$, and neural network weight dimension $d_W$ for each network are summarized.

\begin{table}[H] 
\center
\begin{tabular}{|l||l|l|}
\hline
$d_M$             &  $4,225$ & $37,249$  \\
\hline
\hline (DIPNet, KLE, RS) projected network $r = 8$  &  $1,124$ &  $1,124$ \\
\hline (DIPNet, KLE, RS) projected network $r = 32$  &  $6,500$ &  $6,500$ \\
\hline (DIPNet, KLE, RS) projected network $r = 128$  &  $49,740$ &  $49,740$ \\
\hline FS network  & $442,800$ & $3,745,200$  \\
\hline
\end{tabular}
\caption{Neural network weight dimensions and parameter dimension for different meshes.}
\label{table_summary_of_dimensions_confusion}
\end{table}

For a first set of results we compare the rank $r = 8$ projected networks with the FS network for the coarse mesh. In Figure \ref{confusion_fixed_rank_coarse}, the DIPNet performs better than all three other networks. The KLE network performs about $2\%$ worse in generalization accuracy, both of these networks perform about the same for different initial guesses for the inner dense layer weights. The RS network performs better than the FS network, which suggests that the reduced dimension architecture has benefits in this regime, however it is very sensitive to the initial guesses for the weights as indicated by the one standard deviation error bars. The DIPNet and KLE networks perform significantly better than the RS network which suggests that there is definite upside to choosing input and output bases for these networks.

\begin{figure}[H]
\begin{subfigure}{0.5\textwidth}
\includegraphics[width = \textwidth]{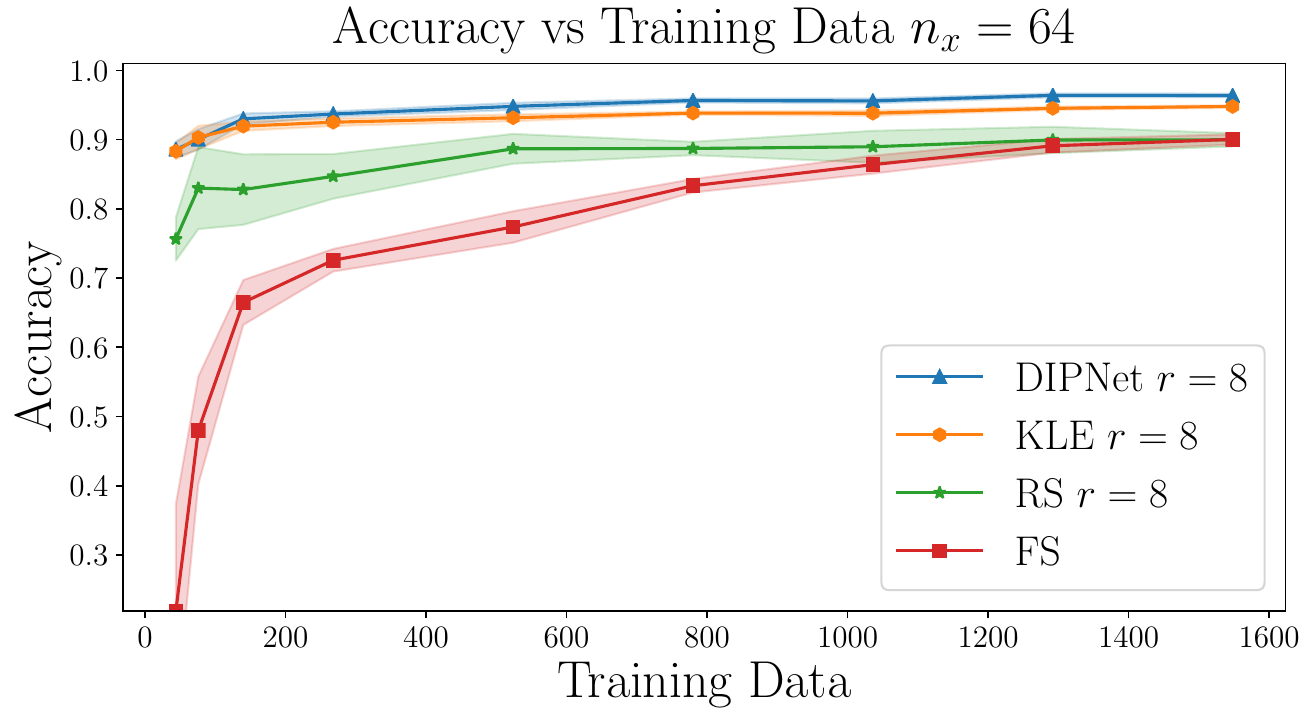}
\end{subfigure}%
\begin{subfigure}{0.5\textwidth}
\includegraphics[width = \textwidth]{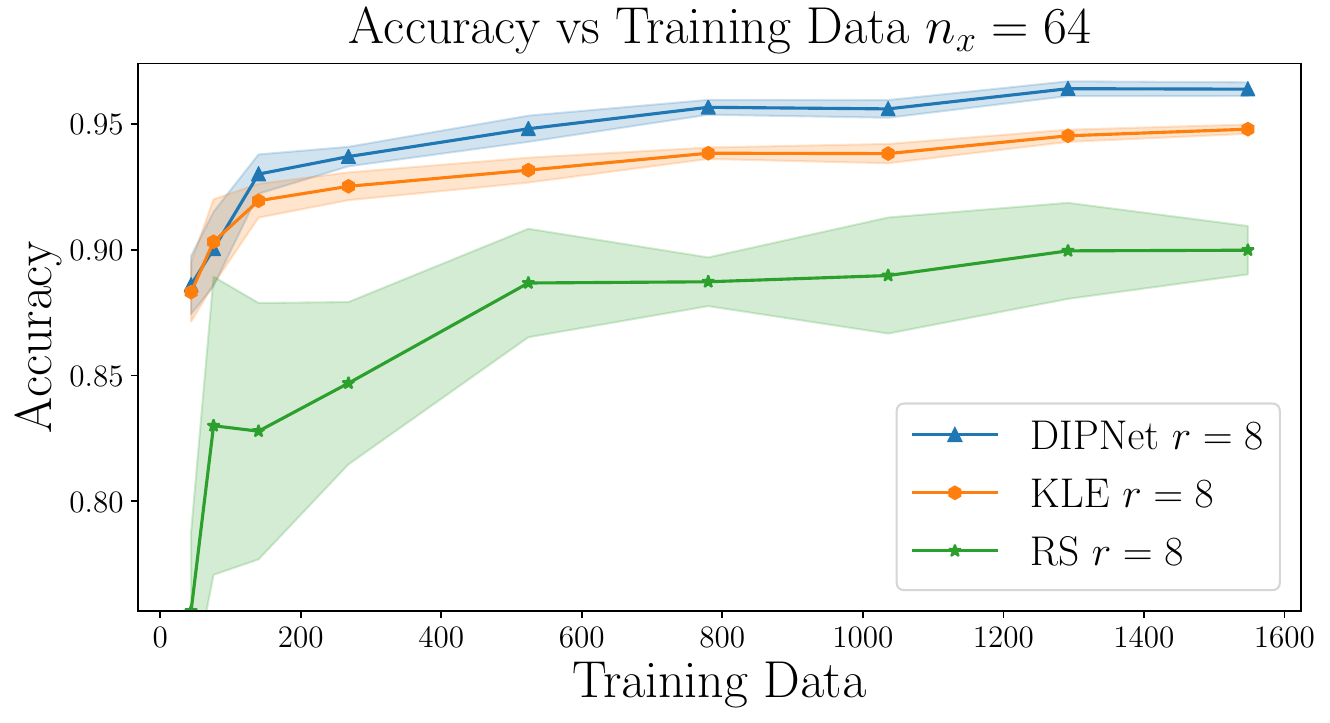}
\end{subfigure}
\caption{Generalization accuracy vs number of training data seen for all networks (left), and just the projected networks (right) for the coarse mesh convection-diffusion-reaction problem.}
\label{confusion_fixed_rank_coarse}
\end{figure}

Figure \ref{confusion_fixed_rank_fine} shows that as the parameter dimension $d_M$ grows, the RS and FS networks become harder to train, while the DIPNet and KLE networks perform comparably well to how they did for the coarse mesh. This suggests that there is mesh independent information that can be encapsulated by the AS, KLE and POD eigenvectors.

\begin{figure}[H]
\begin{subfigure}{0.5\textwidth}
\includegraphics[width = \textwidth]{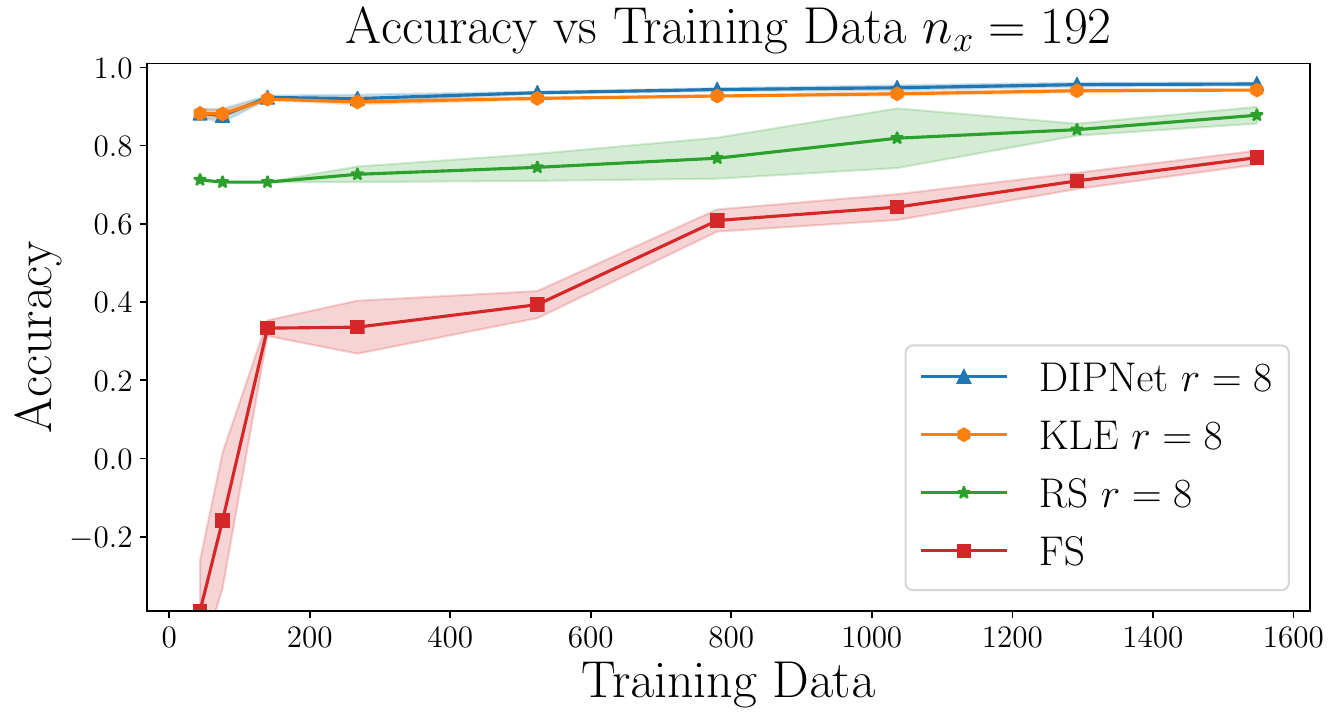}
\end{subfigure}%
\begin{subfigure}{0.5\textwidth}
\includegraphics[width = \textwidth]{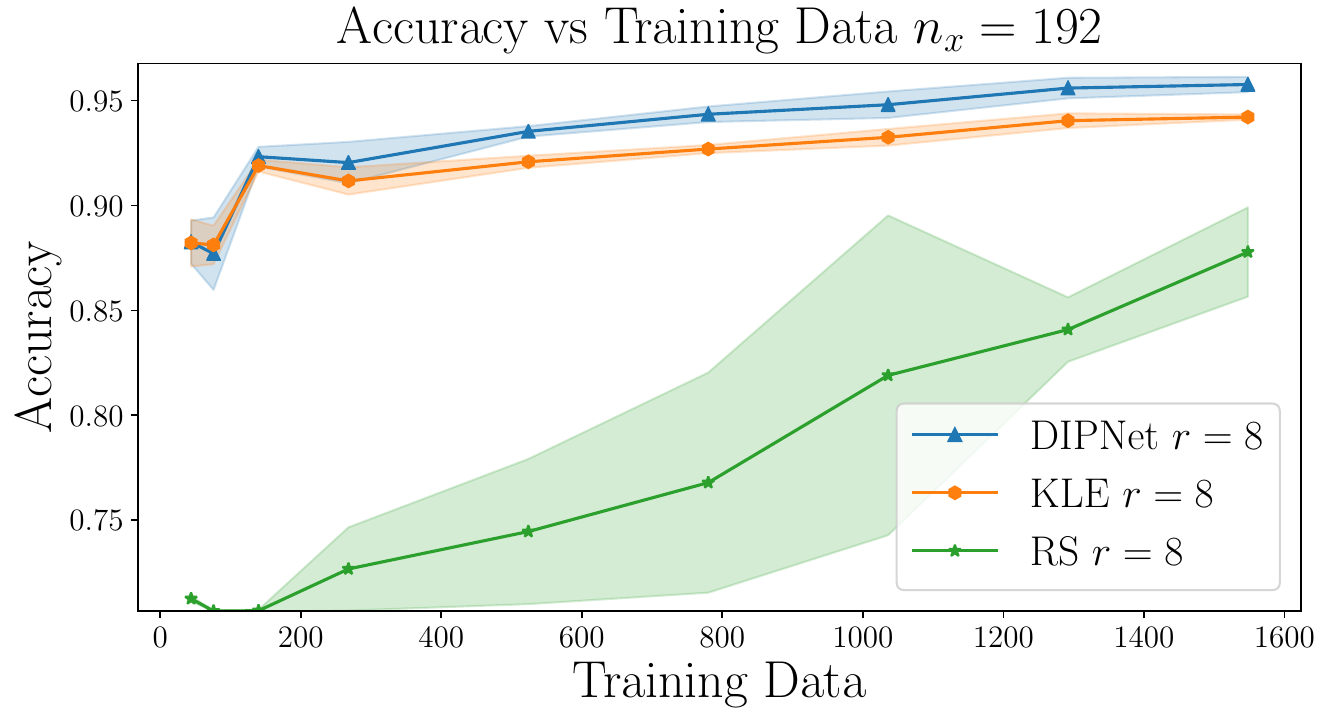}
\end{subfigure}
\caption{Accuracy vs number of training data seen for all networks
  (left), and just the projected networks (right) for the fine mesh
  convection-diffusion-reaction problem.}
\label{confusion_fixed_rank_fine}
\end{figure}

For the next set of numerical results we compare the performance of the DIPNet and KLE projected networks for different choices of the rank $r$. Figure \ref{confusion_fixed_rank_as_kle_small} shows that the advantages of the DIPNet are pronounced when the rank is small, but as the rank $r$ grows the KLE network starts to catch up. While the superior performance of DIPNet in low dimensions agrees with the projection errors observed in Figure \ref{pod_eigs_and_errors_confusion}, the comparable performance of KLE to DIPNet for higher dimensions does not agree. We believe that this has to do with the errors inherent in the neural network representation and training, as discussed in Section \ref{section:discussion_of_errors}. 

As the input bases grow, there will be more overlap between AS and
KLE, and more room for the KLE network to learn patterns captured by
the low dimensional AS basis. Note also that generally, as the input
and output bases grow, the variance with respect to the neural network initial guesses decreases. Similar trends are observed in Figure \ref{confusion_fixed_rank_as_kle_large} for the finer mesh discretization.

\begin{figure}[H]
\begin{subfigure}{0.5\textwidth}
\includegraphics[width = \textwidth]{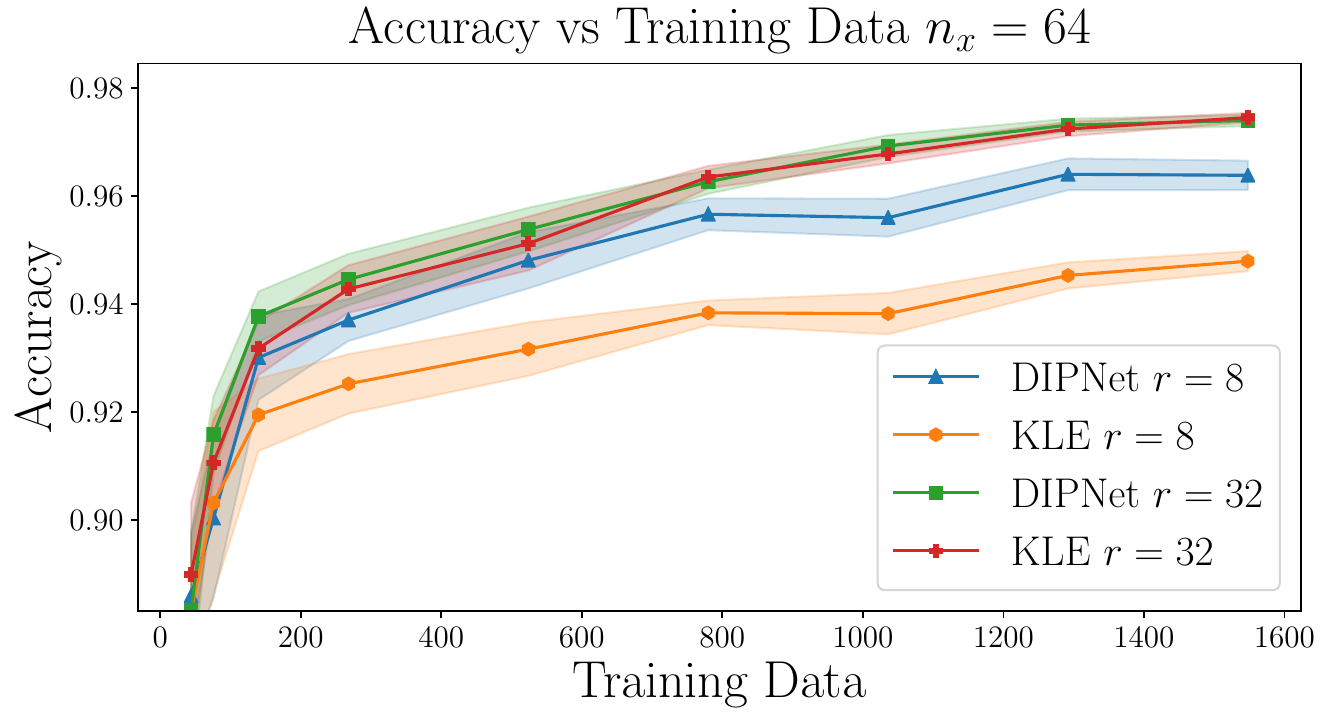}
\end{subfigure}%
\begin{subfigure}{0.5\textwidth}
\includegraphics[width = \textwidth]{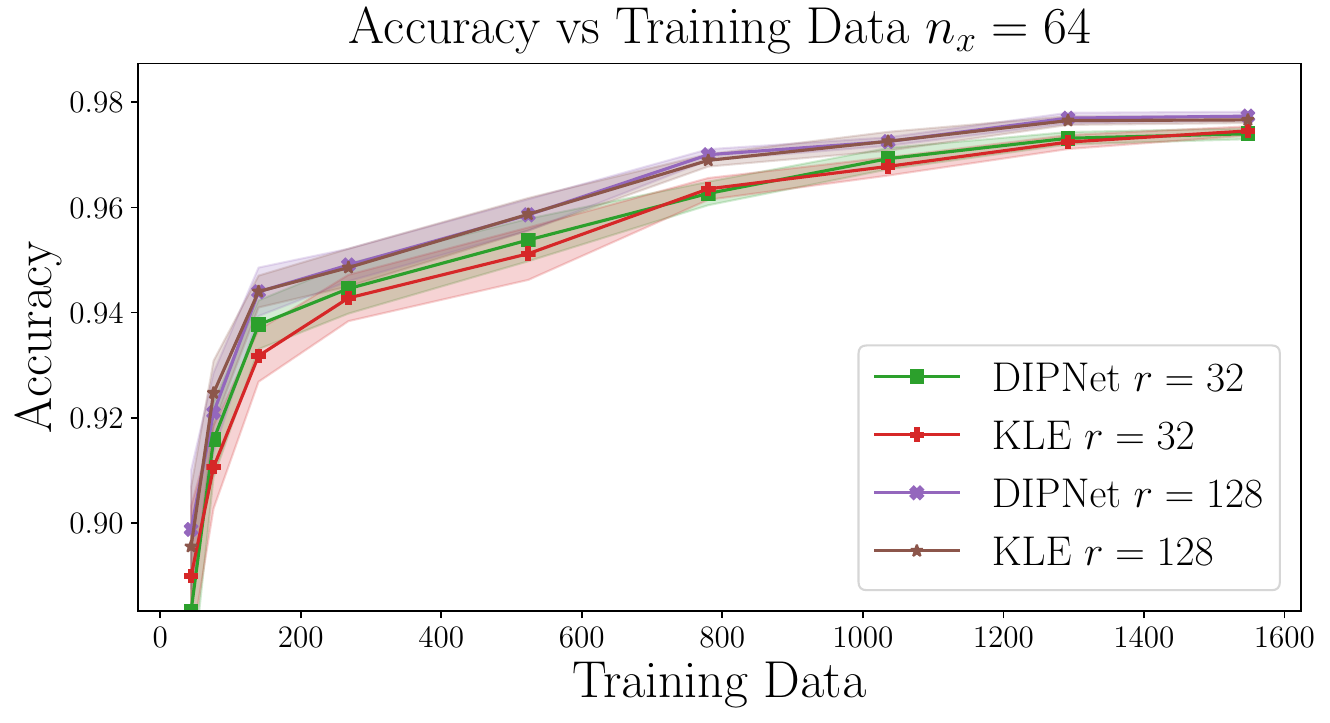}
\end{subfigure}
\caption{Accuracy vs number of training data seen for DIPNet and KLE networks as a function of rank for the coarse mesh convection-diffusion-reaction problem.}
\label{confusion_fixed_rank_as_kle_small}
\end{figure}

\begin{figure}[H]
\begin{subfigure}{0.5\textwidth}
\includegraphics[width = \textwidth]{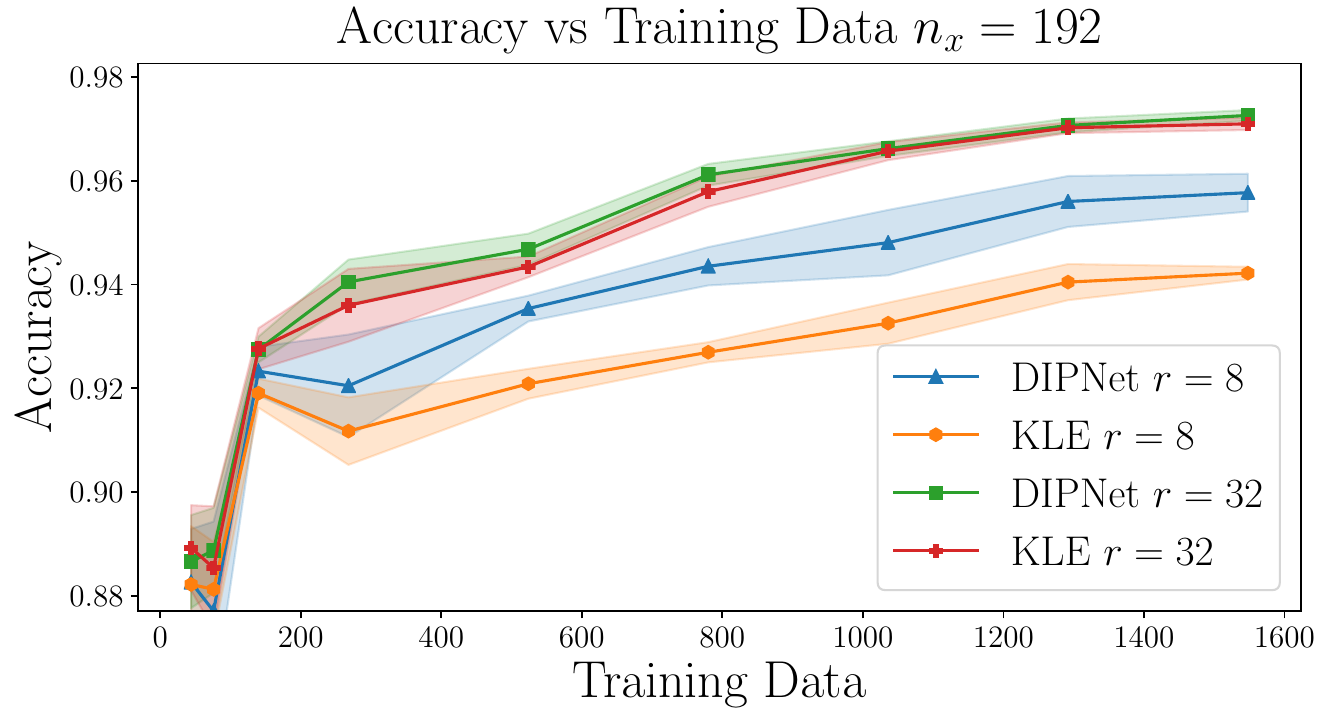}
\end{subfigure}%
\begin{subfigure}{0.5\textwidth}
\includegraphics[width = \textwidth]{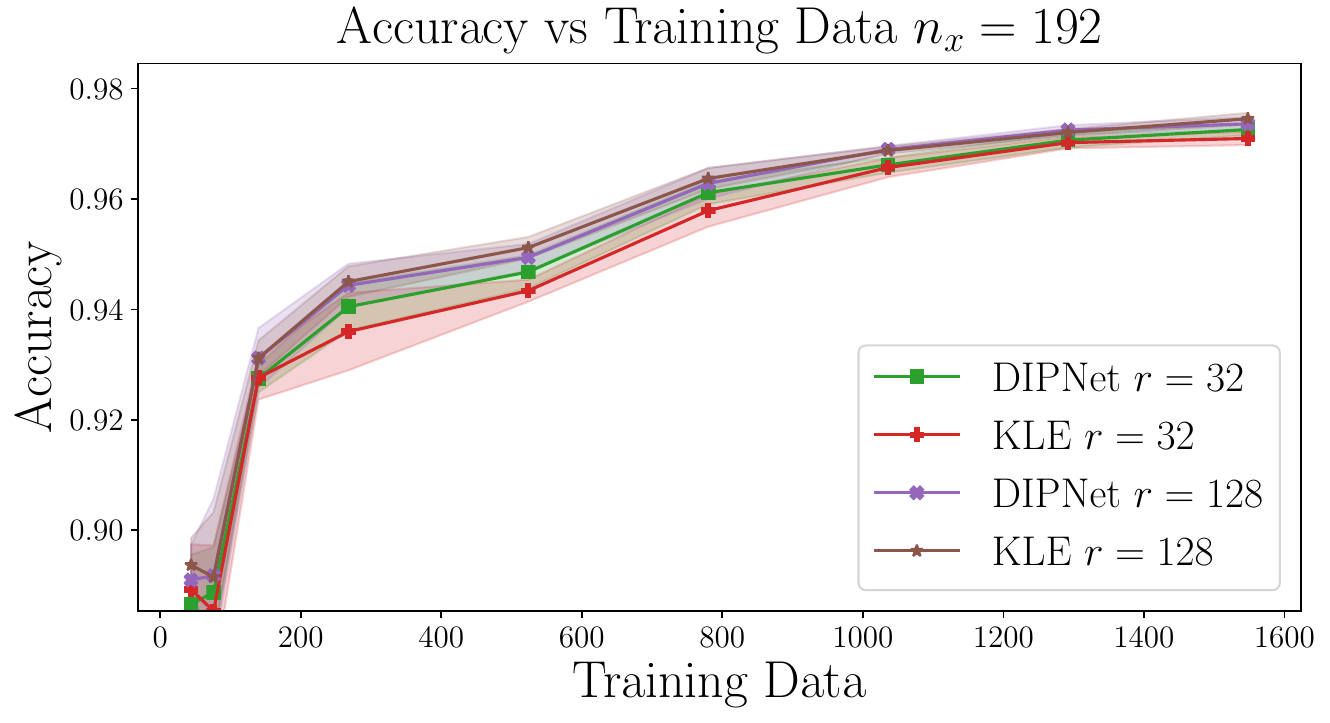}
\end{subfigure}
\caption{Accuracy vs number of training data seen for DIPNet and KLE networks as a function of rank for the fine mesh convection diffusion problem.}
\label{confusion_fixed_rank_as_kle_large}
\end{figure}

\subsection{Helmholtz Problem}

For a second case study, we investigate a 2D \linebreak Helmholtz problem with a
nonlinear dependence on the random field $m$, which represents the
log-prefactor of the wavenumber $k$. The formulation of the problem is
\begin{subequations}
\begin{align}
  - \Delta u  - (k e^m)^2u &= f \quad \text{in } \Omega \\
   \text{PML boundary condition} &\text{ on } \partial \Omega \setminus \Gamma_\text{top} \\
   \nabla u \cdot n &= 0  \text{ on }  \Gamma_\text{top} \\
  q(m) = Bu(m) &= [u(\mathbf{x}^{(i)},m)] \quad \text{at } \mathbf{x}^{(i)} \in \Omega\\
  \Omega &= (0,3)^2.
\end{align}
\end{subequations}

The PML boundary condition on the sides and bottom of the domain
simulates truncation of a semi-infinite domain; waves can only be
reflected on the top surface. Due to the PML domain truncation, the
PDE for the Mat\'{e}rn covariance for the parameter distribution has
Robin boundary conditions to eliminate boundary artifacts
\cite{DaonStadler18}. The problem has a single source at a point
$(0.775,2.85)$, and the $100$ observation points are clustered in a
box $(0.575,0.975)\times (2.75,2.95)$, none of the observation points
coincide with the source. Since the PDE state variable here is a
velocity field, the outputs have dimension $d_Q = 200$ for this
problem. The wave number for this problem is $9.118$. The Gaussian
Mat\'{e}rn distribution $\nu$ is parametrized by $\gamma = 1.0,\delta
= 5.0$. We consider two meshes for this problem again, $n_x = n_y =
64$ and $128$. We again start by investigating the eigenvalue
decompositions for AS, KLE and POD, and the input--output projection
errors given by \ref{input_output_projection_error}.

In Figure \ref{as_kle_spectra_helmholtz}, the AS spectra agree between the two mesh discretization and are roughly mesh independent. Note that for this problem the decay of the dominant modes of the AS spectra compared to KLE are much more pronounced than in Figure \ref{as_kle_spectra_confusion}. The KLE spectra agree between the two mesh discretization and are roughly mesh independent.

\begin{figure}[H]
\begin{subfigure}{0.5\textwidth}
\center
\includegraphics[width = \textwidth]{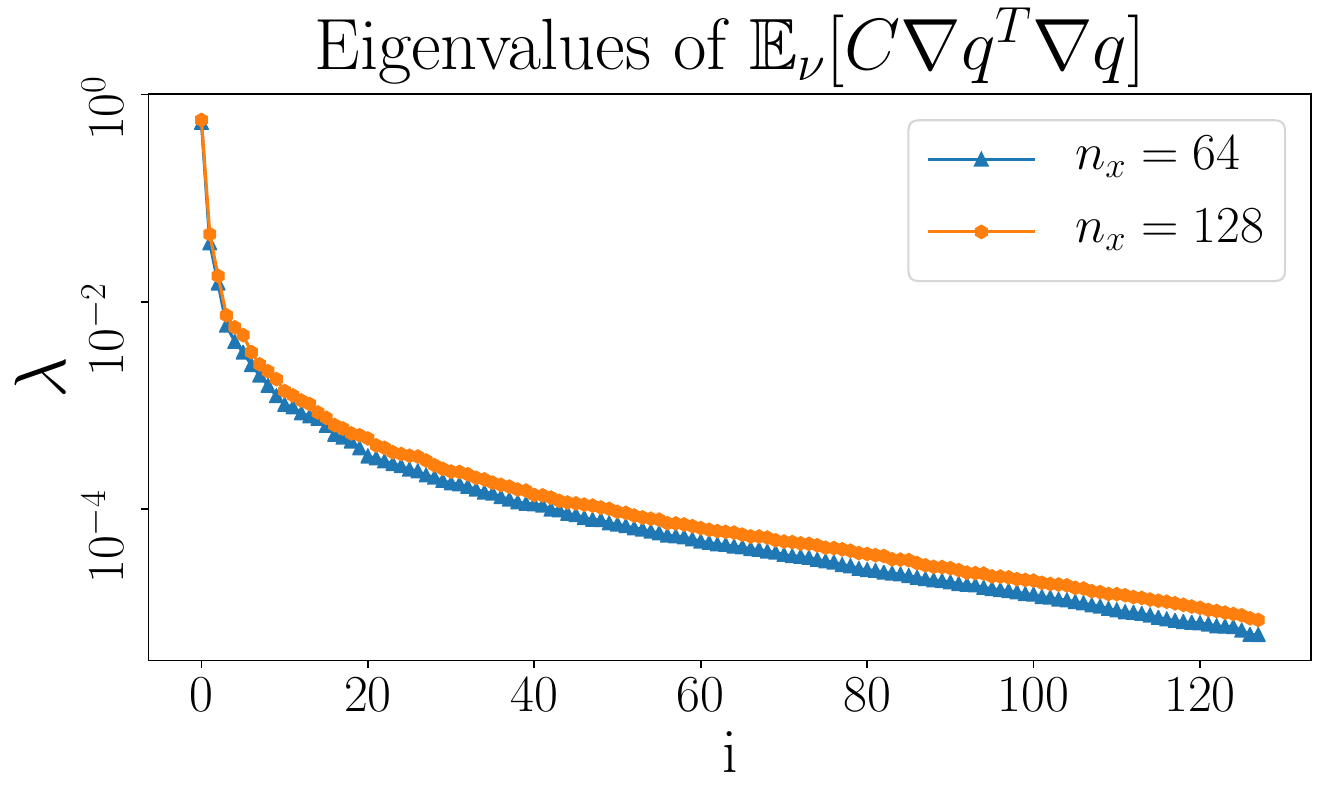}
\end{subfigure}%
\begin{subfigure}{0.5\textwidth}
\center
\includegraphics[width = \textwidth]{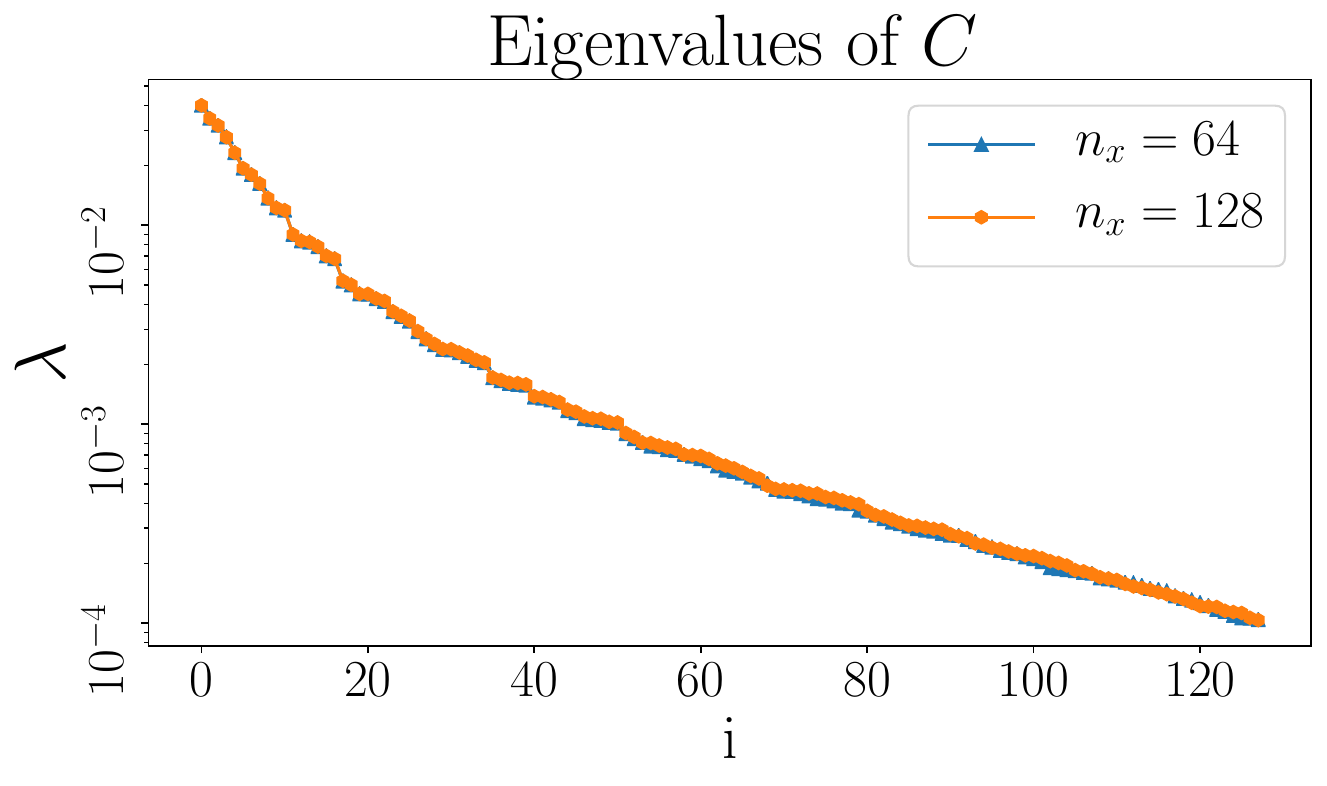}
\end{subfigure}%
\caption{Active subspace and KLE spectra for $\gamma =1.0, \delta = 5.0$}
\label{as_kle_spectra_helmholtz}
\end{figure}

The dominant AS vectors are localized to the part of the domain where the observations are present. The first mode captures strictly locally supported information of the uncertain parameters, and the higher modes start to capture higher frequency effects inherent to the Helmholtz problem. The KLE eigenvectors again correspond to the eigenvectors arising from separation of variables, note that for this problem the PDE forward map is dissimilar to the covariance operator, so this basis is not an optimal representation of the PDE state, and therefore a linear restriction of the PDE state. These eigenvectors again do not pick up local information about the mapping $m \mapsto q$.

\begin{figure}[H] 
\begin{minipage}{0.85\textwidth}
\begin{subfigure}{0.25\textwidth}
\center
\includegraphics[width=\textwidth]{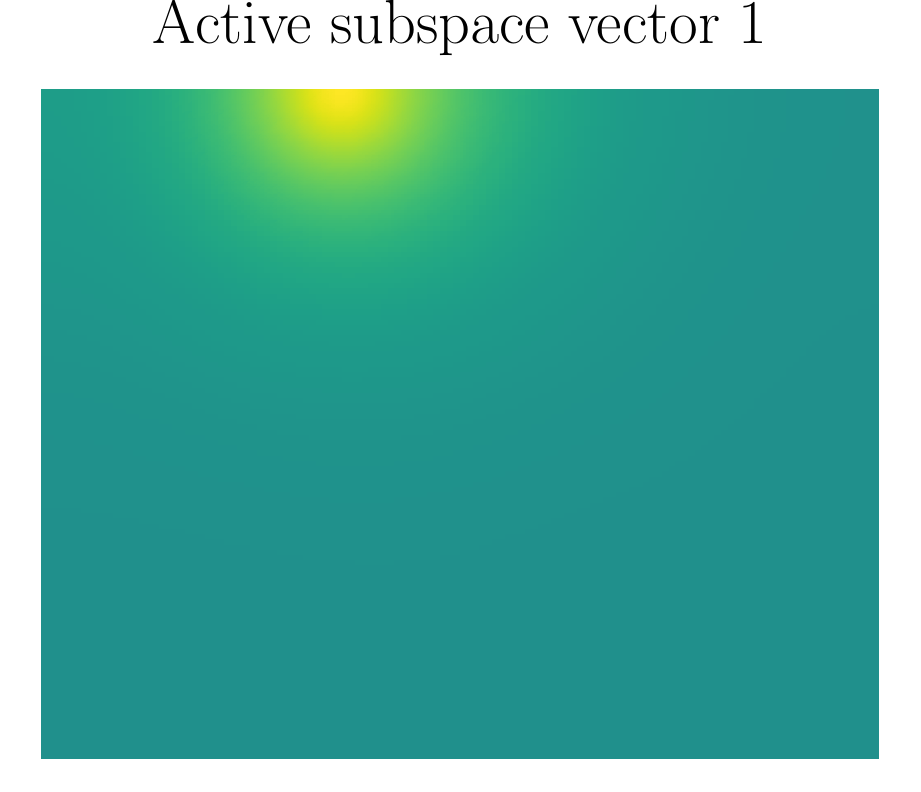}
\end{subfigure}%
\begin{subfigure}{0.25\textwidth}
\center
\includegraphics[width=\textwidth]{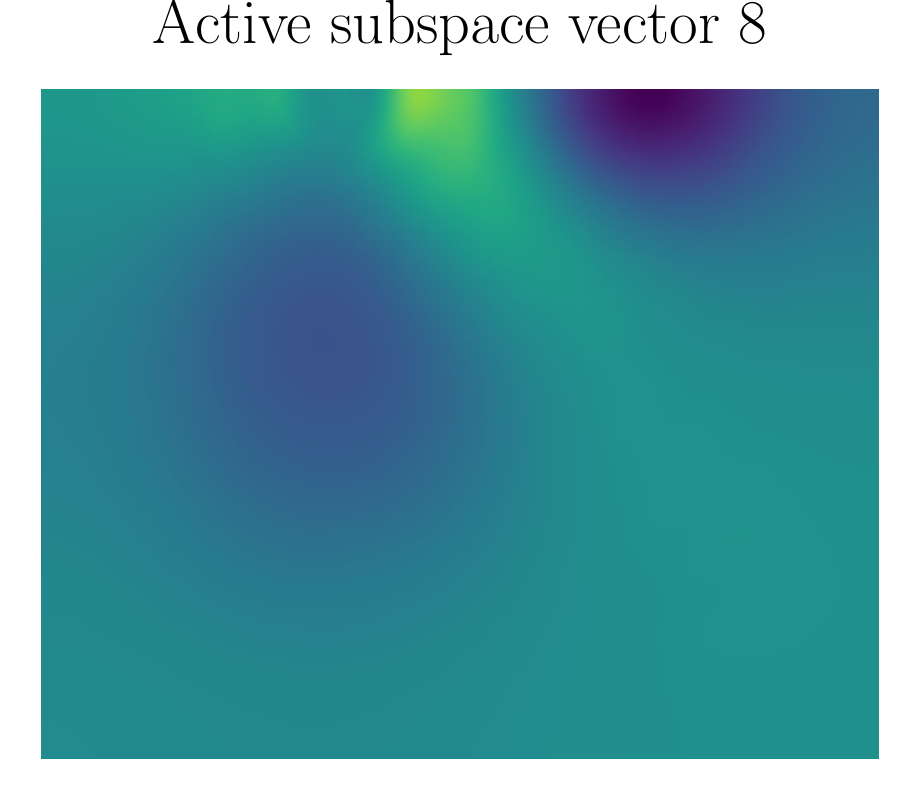}
\end{subfigure}%
\begin{subfigure}{0.25\textwidth}
\center
\includegraphics[width=\textwidth]{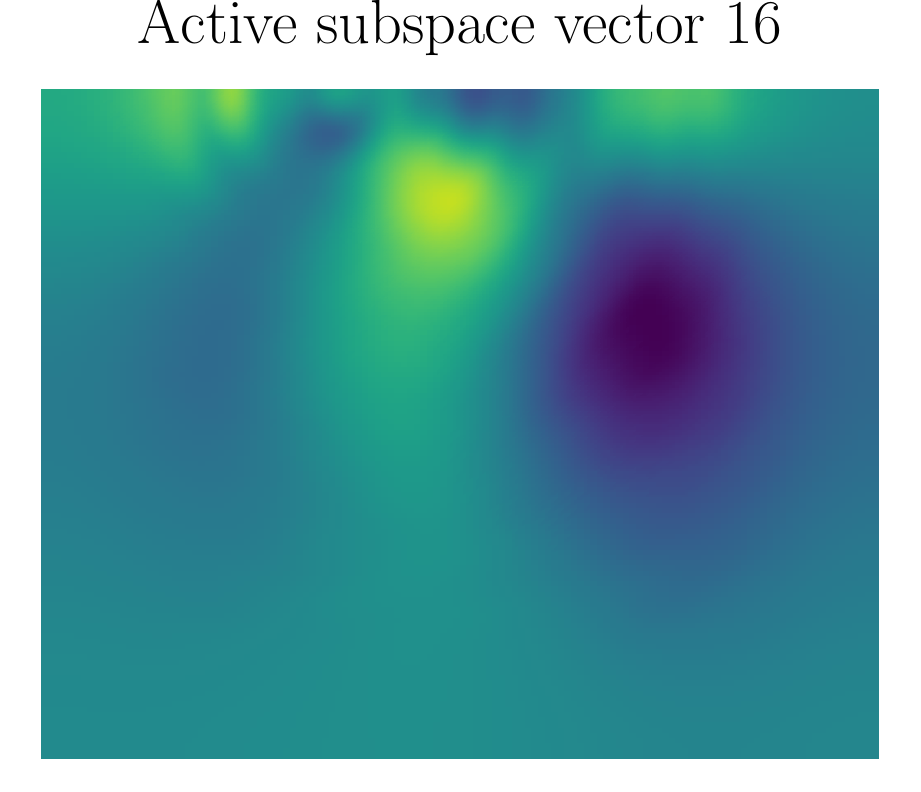}
\end{subfigure}%
\begin{subfigure}{0.25\textwidth}
\center
\includegraphics[width=\textwidth]{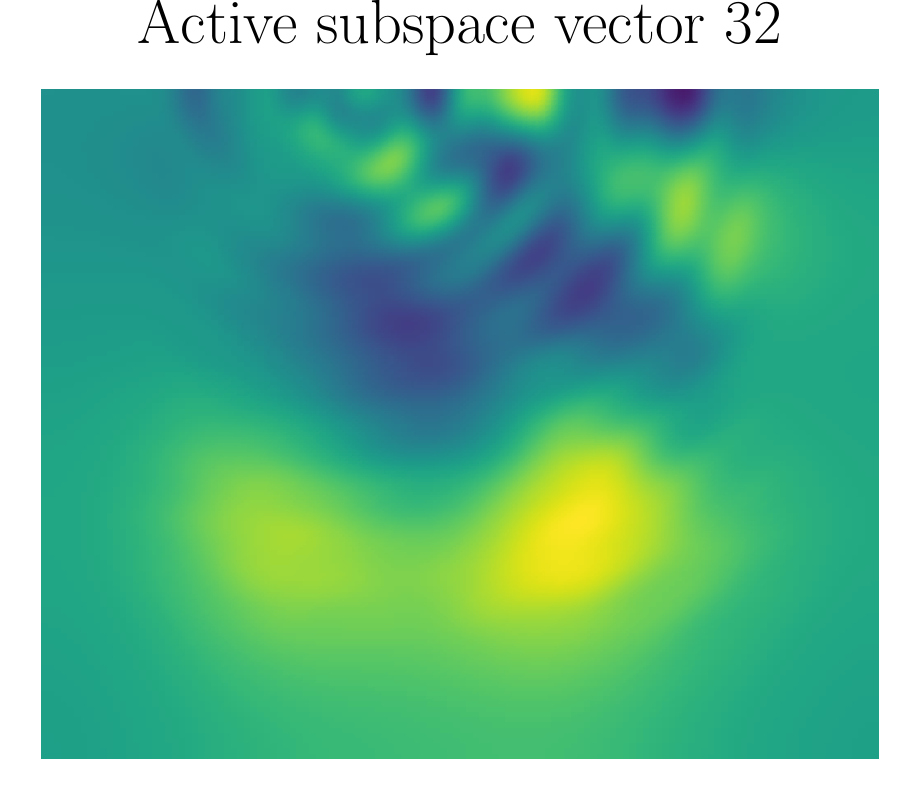}
\end{subfigure}
\begin{subfigure}{0.25\textwidth}
\center
\includegraphics[width=\textwidth]{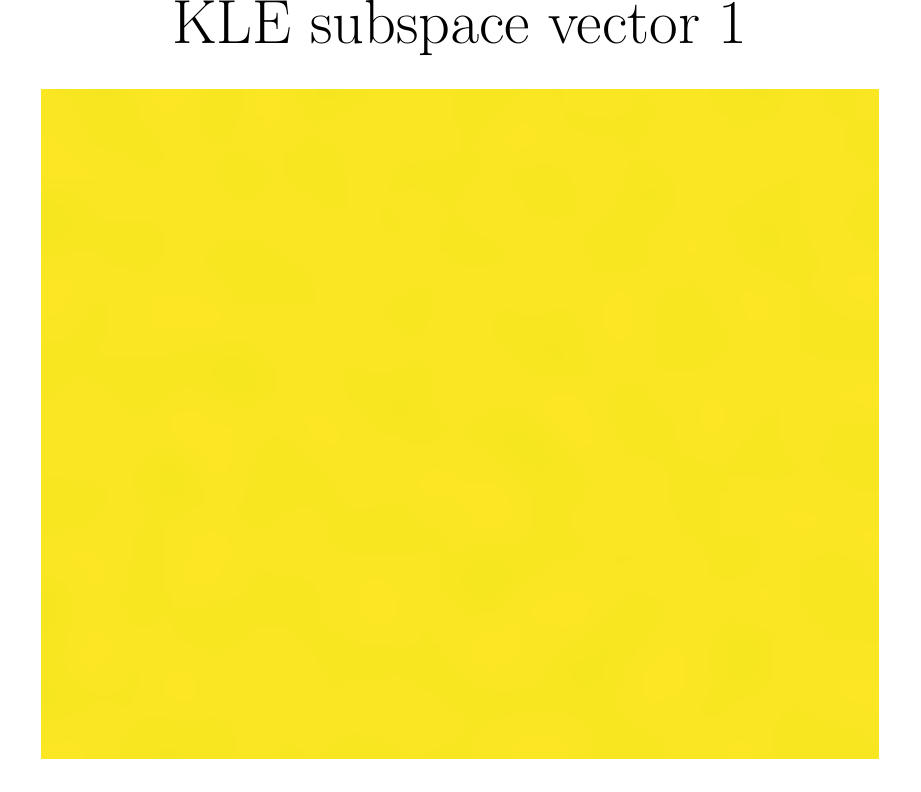}
\end{subfigure}%
\begin{subfigure}{0.25\textwidth}
\center
\includegraphics[width=\textwidth]{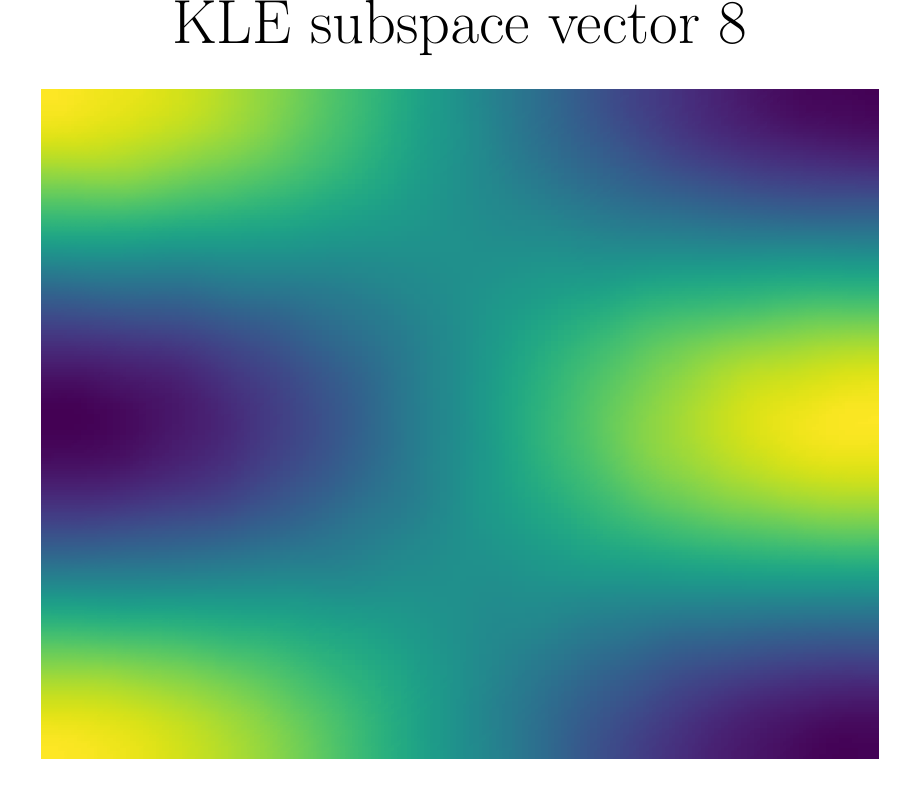}
\end{subfigure}%
\begin{subfigure}{0.25\textwidth}
\center
\includegraphics[width=\textwidth]{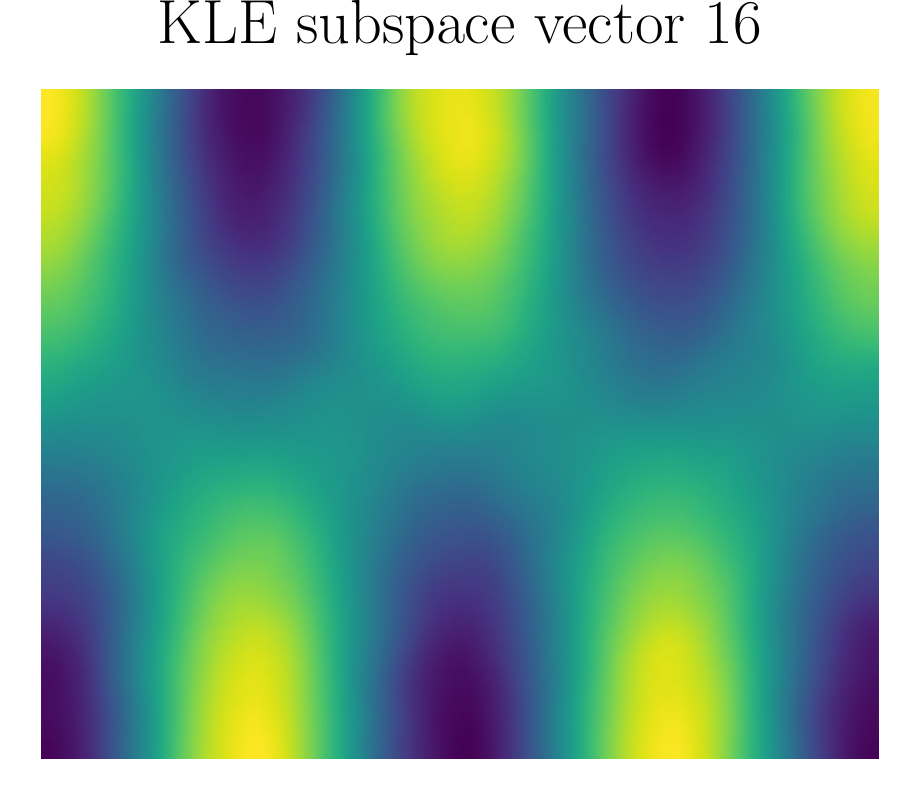}
\end{subfigure}%
\begin{subfigure}{0.25\textwidth}
\center
\includegraphics[width=\textwidth]{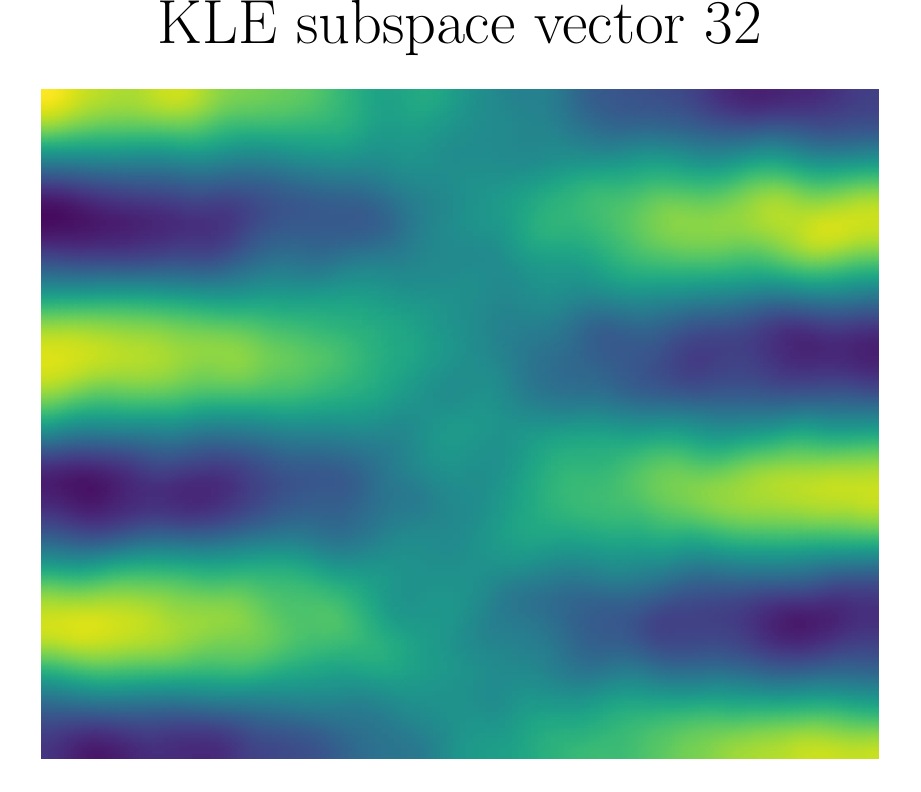}
\end{subfigure}
\end{minipage}%
\begin{minipage}{0.15\textwidth}
\includegraphics[scale = 0.5]{figures/plot_onlycbar.pdf}
\end{minipage}
\caption{AS and KLE eigenvectors for $\gamma =1.0, \delta = 5.0$}
\label{as_kle_vectors_helmholtz}
\end{figure}

In Figure \ref{pod_eigs_and_errors_helmholtz}, the POD spectra agree
in the dominant modes and have similar qualitative decay. They however
begin to diverge in the small modes---even more so than in the
convection-diffusion-reaction problem. The finer discretization
contains more information in the tail of the spectrum. The AS
projector contains more information about the outputs than KLE,
particularly in the first $15-20$ modes, after that the decay rate of
the projection errors are roughly the same.

\begin{figure}[H]
\begin{subfigure}{0.5\textwidth}
\includegraphics[width = \textwidth]{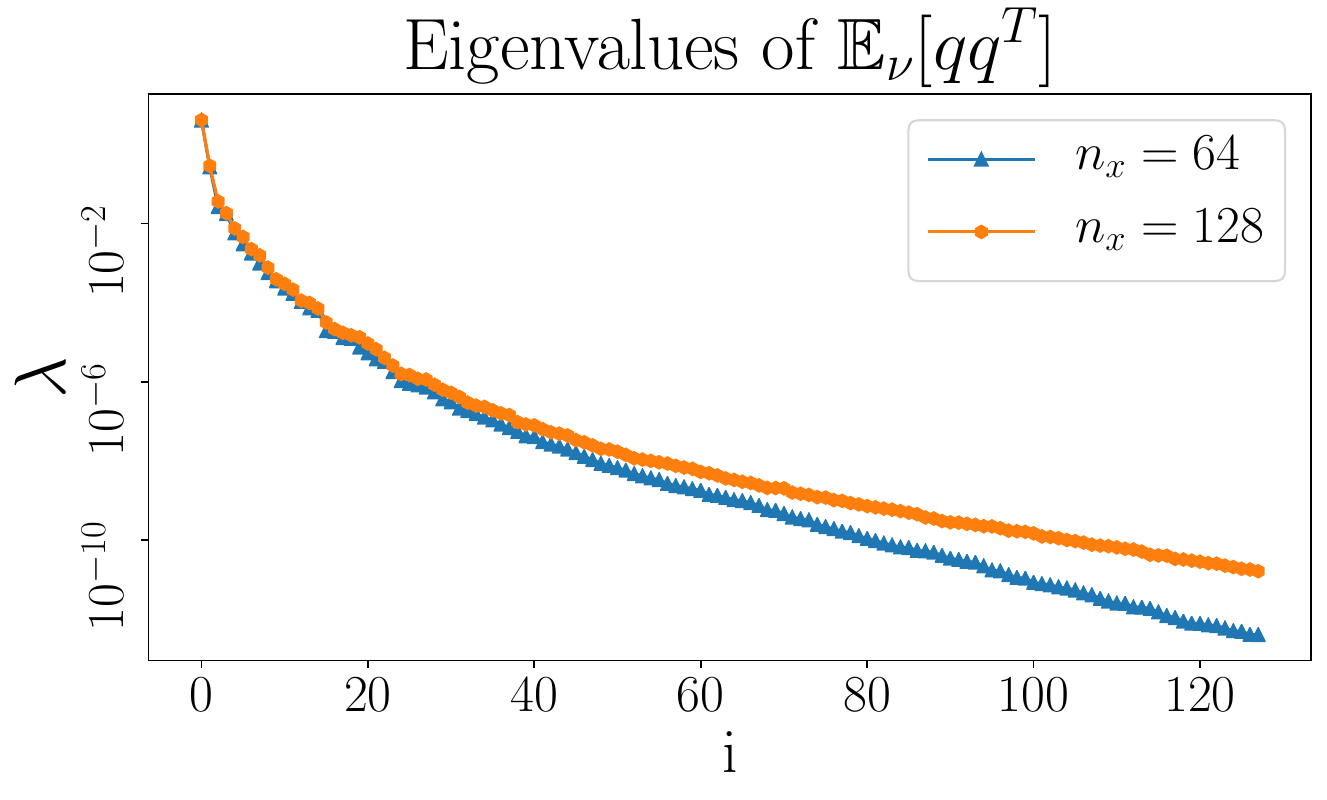}
\end{subfigure}%
\begin{subfigure}{0.5\textwidth}
\includegraphics[width = \textwidth]{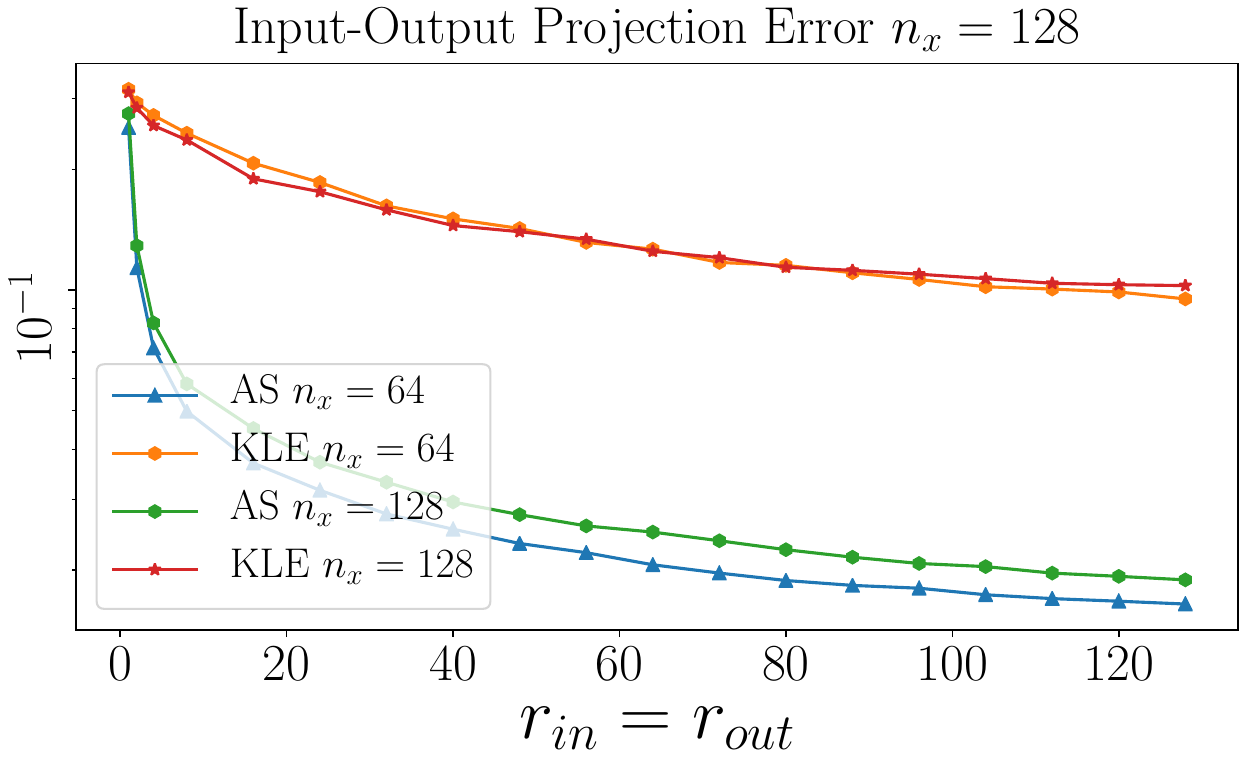}
\end{subfigure}
\caption{Plot of spectra for POD on the two meshes (left), and the input--output projection error (right) for the Helmholtz problem}
\label{pod_eigs_and_errors_helmholtz}
\end{figure}

In what follows we train the projected neural network networks and compare against the FS networks. Below in Table \ref{table_summary_of_dimensions_helmholtz}, dimensions for the input parameters $d_M$, and $d_W$ for each network are summarized.

\begin{table}[H]
\center
\begin{tabular}{|l||l|l|}
\hline
$d_M$             &  $4,225$ & $16,641$  \\
\hline
\hline (DIPNet, KLE, RS) projected network $r = 8$  &  $2,024$ &  $2,024$ \\
\hline (DIPNet, KLE, RS) projected network $r = 32$  &  $9,800$ &  $9,800$ \\
\hline (DIPNet, KLE, RS) projected network $r = 64$  &  $25,544$ &  $25,544$ \\
\hline FS network  & $925,600$ & $3,408,800$  \\
\hline
\end{tabular}
\caption{Neural network weight dimensions and parameter dimension for the different meshes used.}
\label{table_summary_of_dimensions_helmholtz}
\end{table}

For a first set of results we again compare the rank $r = 8$ projected networks with the FS network for the coarse mesh. In Figure \ref{helmholtz_fixed_rank_coarse}, the DIPNet performs better than all three other networks. For this problem the KLE network performs about $10\%$ worse in generalization accuracy, both of these networks perform about the same for different initial guesses for the inner dense layer weights. The RS and KLE networks performs better than the FS network in the low data regime, but the FS network begins to outperform these two networks when more training data are available. The FS network nearly catches up to the DIPNet $r = 8$ network in the high data limit. The benefit of the DIPNet and FS networks over the KLE and RS networks seems to be their ability to resolve more oscillatory information than the KLE and RS networks. The FS network is able to do this when more data are available, since it has many weights to be fit, the DIPNet can do this in both the limited data regime and when more data are available, since the AS basis is able to represent these modes. The KLE network represents smooth modes of the parameter distribution covariance, which are not as useful when the parametric mapping is dominated by highly oscillatory modes as in this example. 

The benefit of KLE for the convection-diffusion-reaction problem is that the PDE problem was similar to the PDE that shows up in the Gaussian Mat\'{e}rn covariance matrix. The limitations of KLE show up for the Helmholtz problem, where the forward PDE mapping is dissimilar from the parameter distribution covariance operator $C$.

\begin{figure}[H]
\begin{subfigure}{0.5\textwidth}
\includegraphics[width = \textwidth]{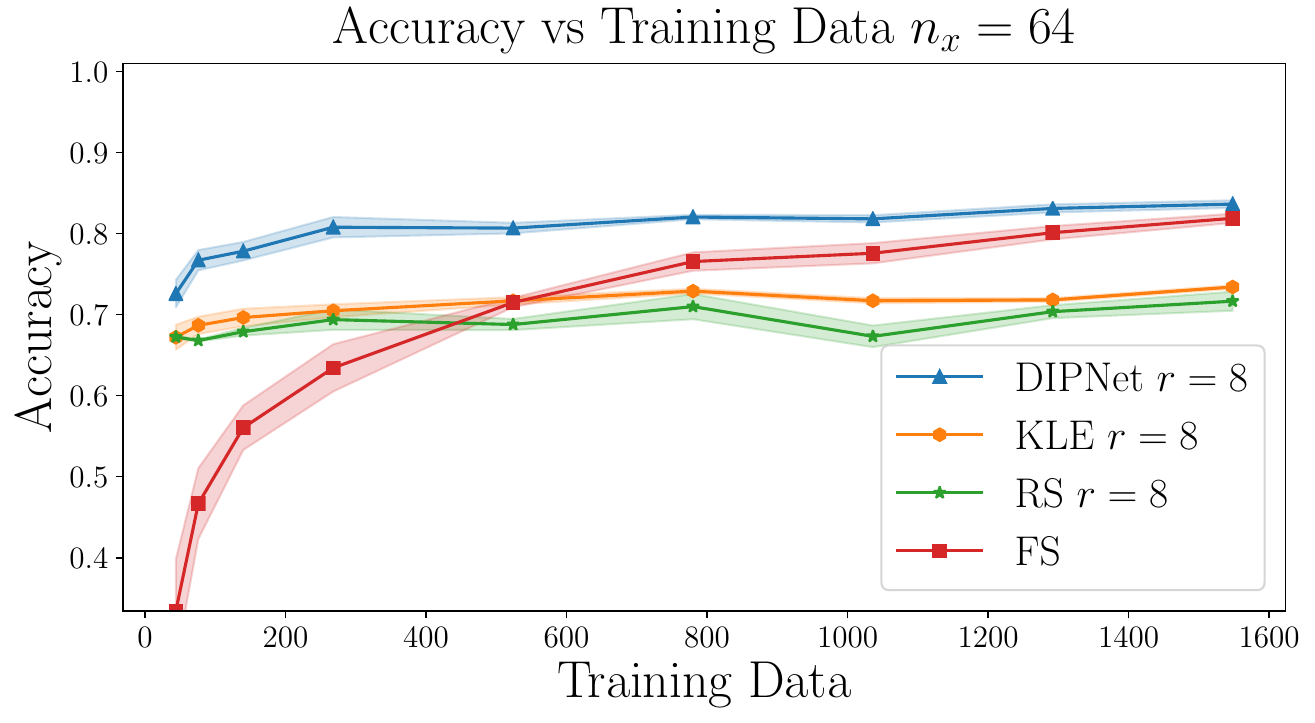}
\end{subfigure}%
\begin{subfigure}{0.5\textwidth}
\includegraphics[width = \textwidth]{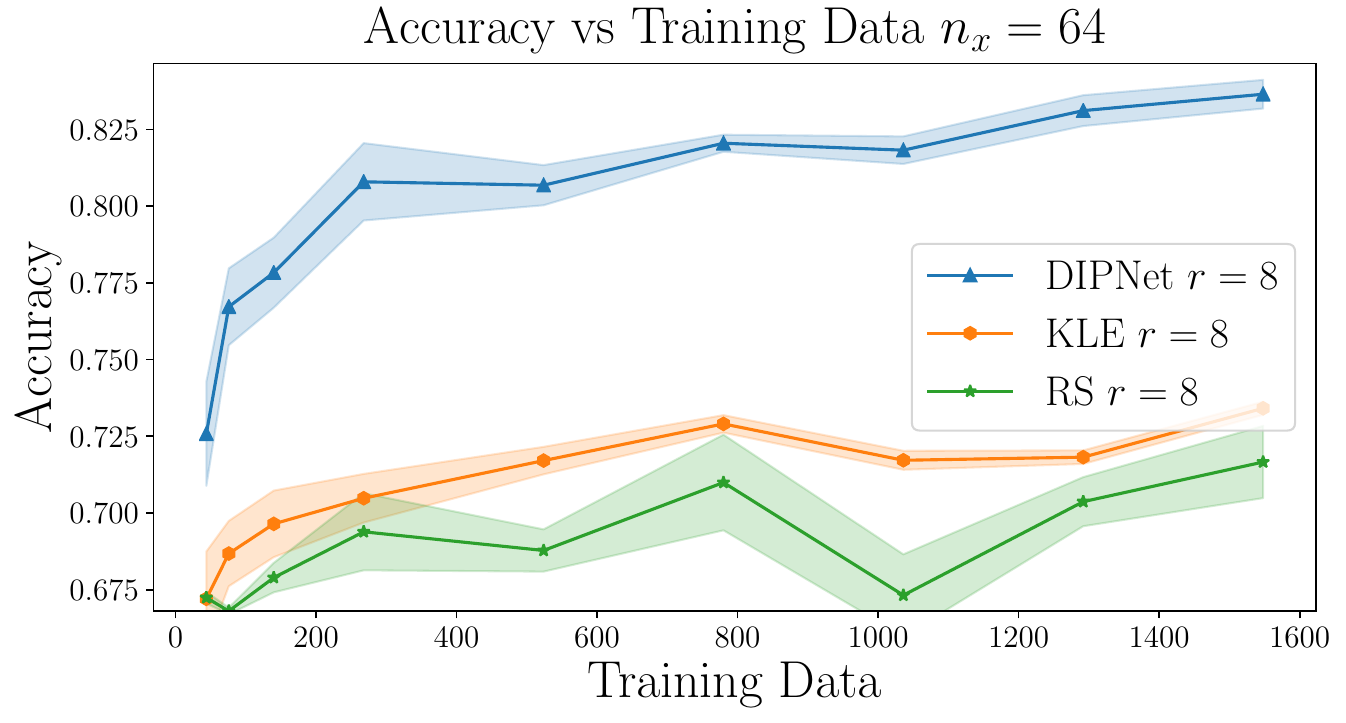}
\end{subfigure}
\caption{Accuracy vs number of training data seen for all networks (left), and just the projected networks (right) for the coarse mesh Helmholtz problem.}
\label{helmholtz_fixed_rank_coarse}
\end{figure}

Figure \ref{helmholtz_fixed_rank_fine} shows that as the parameter dimension $d_M$ grows, the RS and FS networks become harder to train, while the DIPNet and KLE networks perform comparably well to how they did for the coarse mesh. The FS network again nearly catches up the the DIPNet $r = 8$ in the high data limit.

\begin{figure}[H]
\begin{subfigure}{0.5\textwidth}
\includegraphics[width = \textwidth]{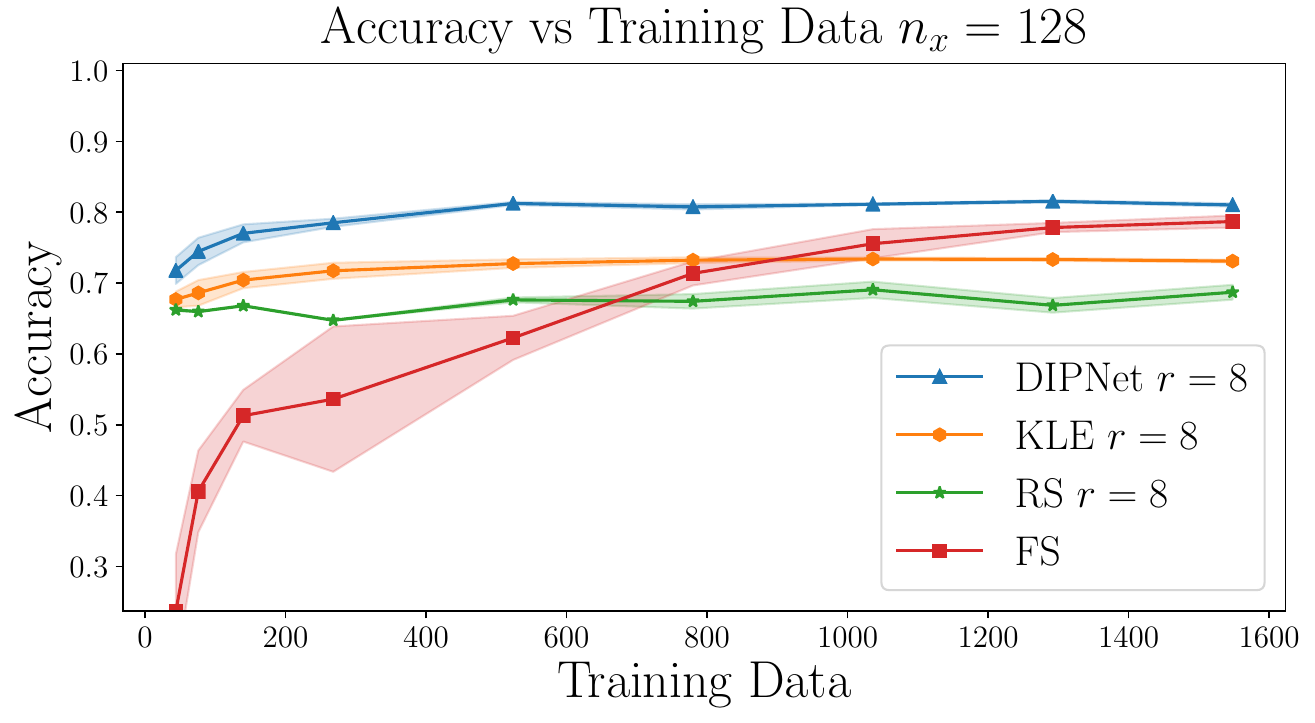}
\end{subfigure}%
\begin{subfigure}{0.5\textwidth}
\includegraphics[width = \textwidth]{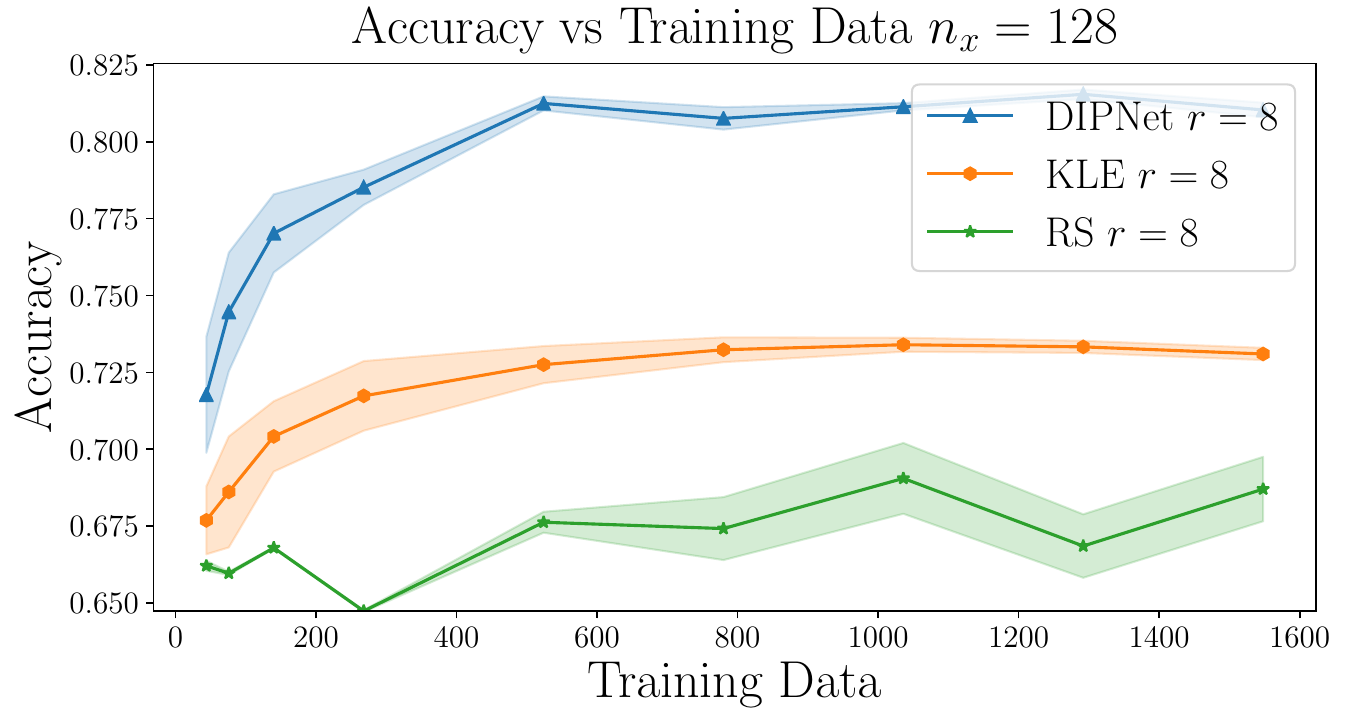}
\end{subfigure}
\caption{Accuracy vs number of training data seen for all networks (left), and just the projected networks (right) for the fine mesh Helmholtz problem.}
\label{helmholtz_fixed_rank_fine}
\end{figure}

For the next set of numerical results we again compare the performance of the DIPNet and KLE projected networks for different choices of the rank $r$. Figure \ref{helmholtz_fixed_rank_as_kle_small} shows that for this problem the DIPNet outperforms the KLE networks consistently. The low dimensional ($r=8$) DIPNet performs well for low data, but sees reduced performance as the data dimension improves. The DIPNets see benefited performance when the rank increases to $32$ and $64$. Both KLE and DIPNet observed reduced accuracy for $r = 128$. The Helmholtz data are evidently harder to fit than the convection-diffusion-reaction data, which are less oscillatory as seen by the AS eigenvectors for the two problems (Figures \ref{as_kle_vectors_confusion} and \ref{as_kle_vectors_helmholtz}). When the dimension of the neural networks grow sufficiently large the difficulties of neural network training begin to dominate the benefits of better approximation in larger reduced bases. 

Similar trends are observed in Figure \ref{helmholtz_fixed_rank_as_kle_large}.
\begin{figure}[H]
\begin{subfigure}{0.5\textwidth}
\includegraphics[width = \textwidth]{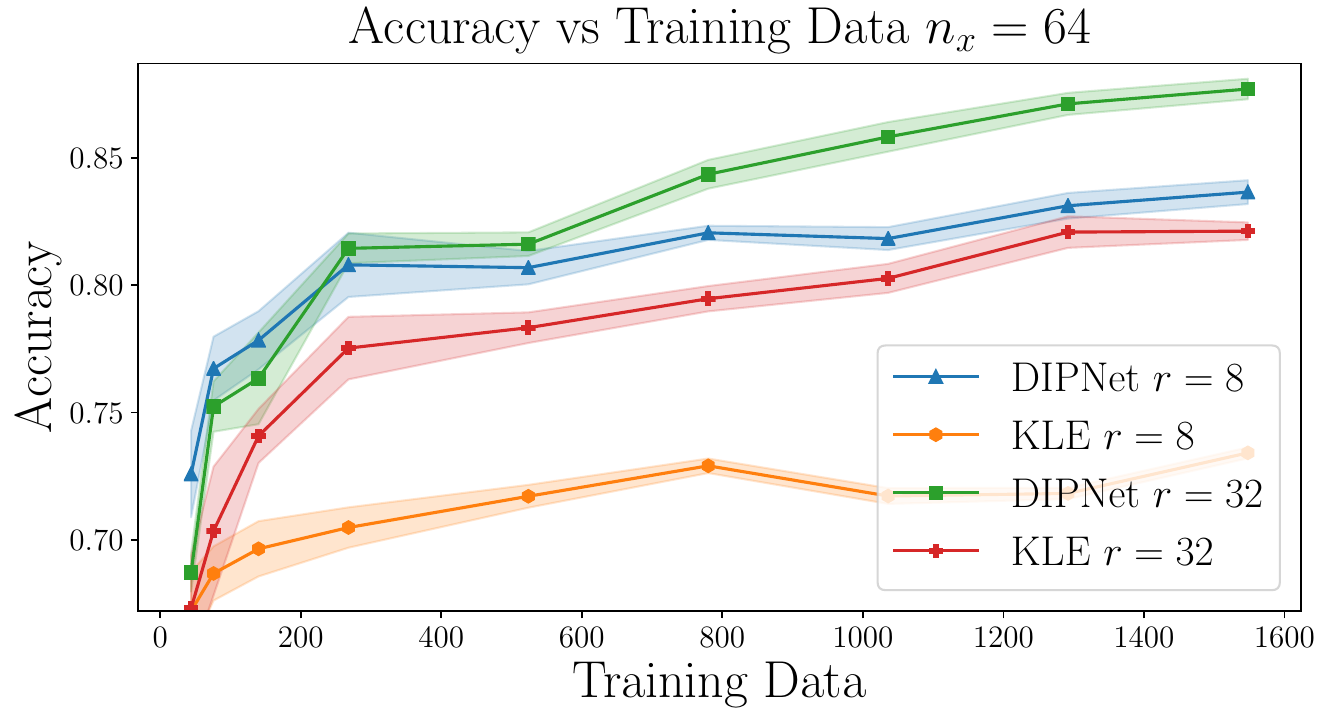}
\end{subfigure}%
\begin{subfigure}{0.5\textwidth}
\includegraphics[width = \textwidth]{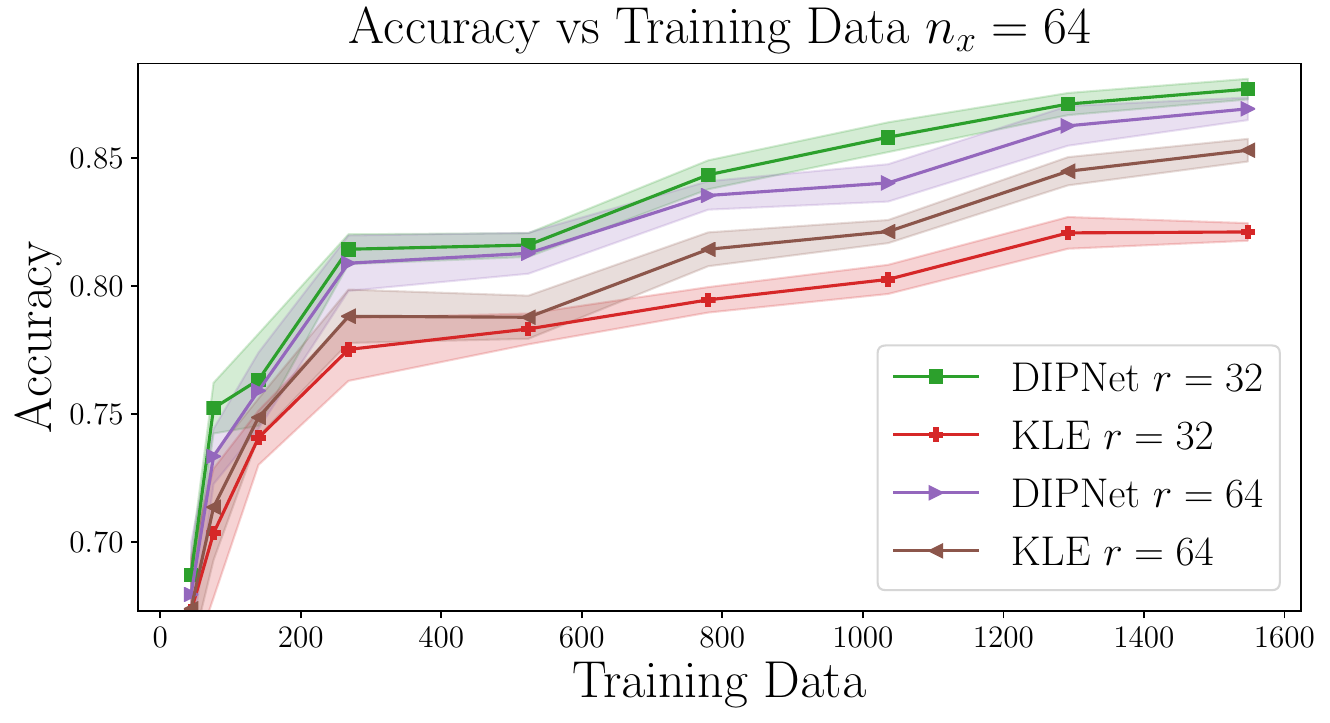}
\end{subfigure}
\caption{Accuracy vs number of training data seen for DIPNet and KLE networks as a function of rank for the coarse mesh Helmholtz problem.}
\label{helmholtz_fixed_rank_as_kle_small}
\end{figure}

\begin{figure}[H]
\begin{subfigure}{0.5\textwidth}
\includegraphics[width = \textwidth]{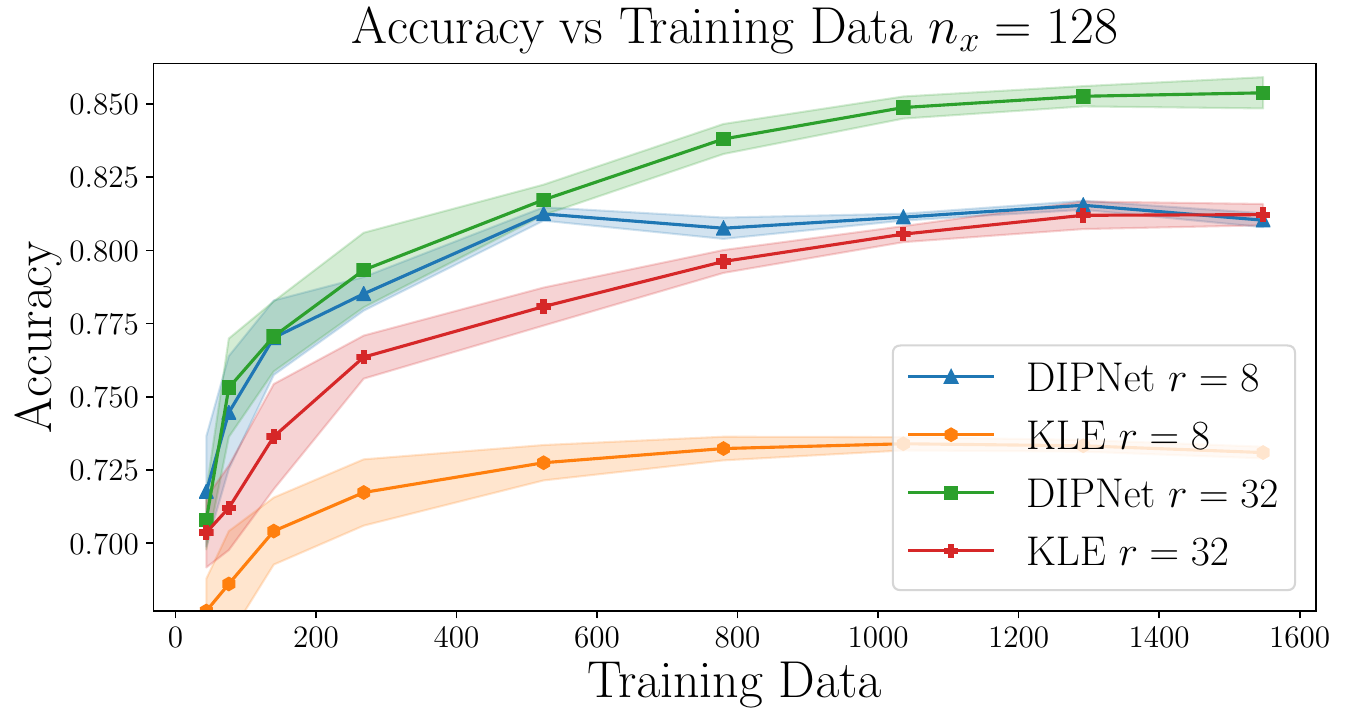}
\end{subfigure}%
\begin{subfigure}{0.5\textwidth}
\includegraphics[width = \textwidth]{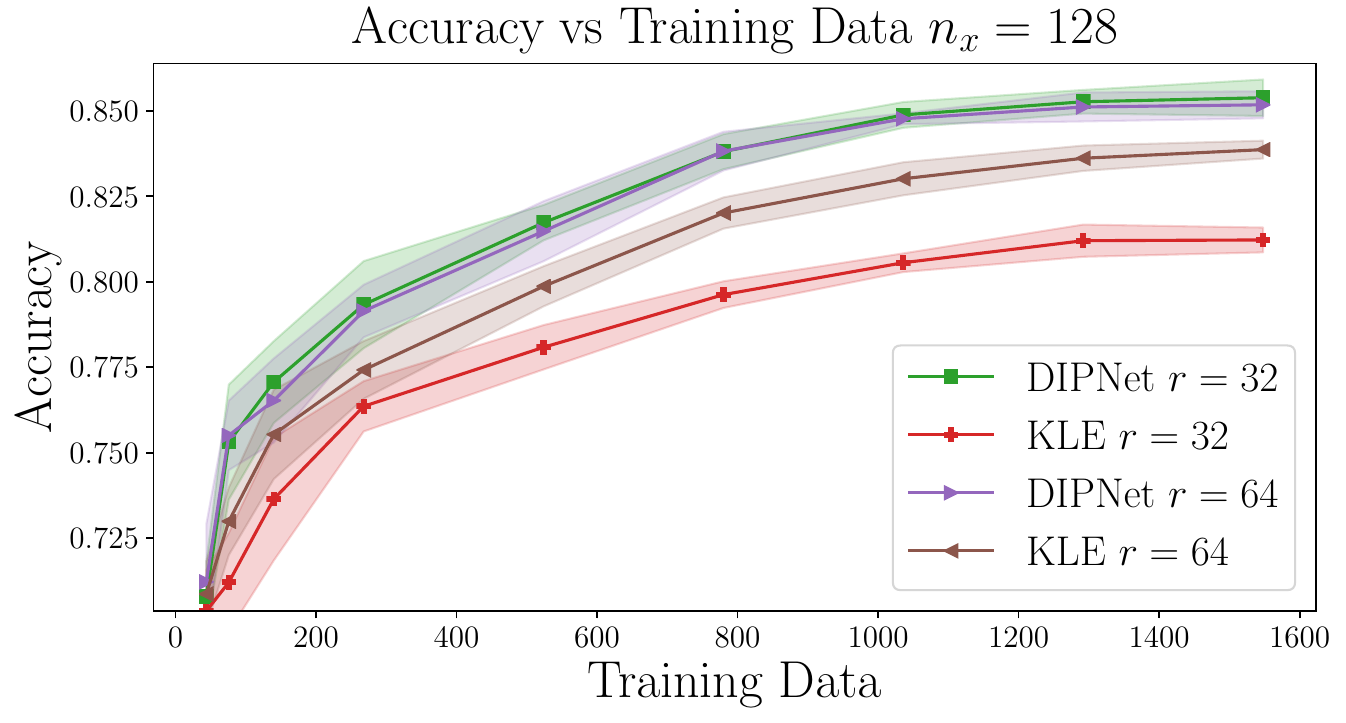}
\end{subfigure}
\caption{Accuracy vs number of training data seen for DIPNet and KLE networks as a function of rank for the fine mesh Helmholtz problem.}
\label{helmholtz_fixed_rank_as_kle_large}
\end{figure}

In these numerical results we compared fixed ranks for the inputs and the outputs, in order to have a straightforward comparison between the DIPNet and KLE networks. In general we recommend the use of the DIPNet, which uses explicit information of the parametric map to construct an informed subspace of the input. The decay of the AS and POD spectra alone can be used to decide the ranks of the input and output projectors for the networks, with the projection error result from Proposition \ref{prop_input_output_error_bound} in mind for this rank choice. This result is less useful when using the KLE spectral decay in choosing a network structure, since this additionally requires knowledge of the Lipschitz constance $L$ for the mapping $m \mapsto q$. Further, numerical results demonstrate that when the PDE mapping does not look like the covariance operator, KLE does not perform much better than random projectors, while the DIPNet performed well for both numerical experiments.

\subsection{Software}

Code for the numerical experiments can be found in \linebreak
hIPPYflow\cite{hippyflow}, a
Python library for parametric PDE based dimension reduction and
surrogate modeling based on
hIPPYlib\cite{VillaPetraGhattas2018,VillaPetraGhattas20}, a
Python library for adjoint based inference problems
. Neural network training was carried out
by
hessianlearn\cite{hessianlearn},
a Python library for Hessian based stochastic optimization in
TensorFlow and Keras
\cite{AbadiAgarwalBarhamEtAl2016,OLearyRoseberryAlgerGhattas2019,OLearyRoseberryAlgerGhattas2020}.

\section{Conclusions}

In this work we have presented a framework for the dimension-independent
construction of derivative-informed projected neural network (DIPNet)
surrogates for parametric maps governed by partial differential
equations.
The need for these surrogates arises in many-query problems for
expensive-to-solve PDEs characterized by high-dimensional parameters.
Such settings are challenging for surrogate construction, since the
expense of the PDE solves implies that a limited number of data can be
generated for training, far fewer than typical data-hungry surrogates
require in high dimensions.
In order to address these challenges, we present
projected neural networks that use derivative (Jacobian) information
for input dimension reduction and principal subspaces of the output to
construct parsimonious networks architectures. We advocate the use
of Jacobian-informed active subspace (AS) projection for the
parameters, and proper orthogonal decomposition (POD) projection for
the outputs, and motivate the strategy with analysis. 

In numerical experiments we compare this DIPNet strategy against a
similar strategy that instead uses Karhunen Lo\`{e}ve Expansion (KLE)
for the input dimension \cite{BhattacharyaHosseiniKovachki2020},
conventional full space dense networks, and similarly architected
projected neural networks using random projectors for the input and
output bases. The full space network performed poorly in the low data
regime, demonstrating the need for projected neural networks in high
dimensional surrogate problems with few training data. The DIPNet and
KLE-to-POD networks outperformed the random subspace projected
network, showing that it is important to choose basis vectors for the
input and output representation in the neural network that exploit the
structure of the parametric PDE-based map.

The KLE-to-POD strategy worked well for the
convection-diffusion-reaction problem, but not as well for the Helmholtz
problem. On the other hand, the DIPNet strategy
performed well in both numerical examples, consistently outperforming
all of the other neural networks. This makes sense in light of the
projection error plots (Figure \ref{pod_eigs_and_errors_confusion} and
Figure \ref{pod_eigs_and_errors_helmholtz}); in the
convection-diffusion-reaction case the KLE basis was able to reduce
the error similar to AS, but in the Helmholtz case the KLE basis was
unable to reduce the projection error to any significant
degree. Figures \ref{as_kle_vectors_confusion} and
\ref{as_kle_vectors_helmholtz} demonstrate that in both cases the AS
eigenvectors were able to resolve key localized information, in
particular highly-oscillatory smaller length scales for the Helmholtz
problem, while the KLE eigenvectors resembled typical elliptic PDE
eigenfunctions as expected.

These projectors used in DIPNet are 
infinite-dimensionally consistent, meaning that the neural network learns
the true information for the PDE and not artifacts of the discretization. 
This strategy allows for extensions to transfer learning and multi-fidelity
methods.

In future work we will employ these surrogates in many-query settings
such as Bayesian inference, optimal experimental design, and
stochastic optimization. The neural network architectures in this work
were kept simple in order to limit the sources of errors (more
complicated neural networks are harder to train). There is
considerable opportunity to tune the parametrizations of the
input--output projected networks to achieve higher performing
networks. Nevertheless, the present work demonstrates the power of
our general approach.

\bibliographystyle{siamplain}
\bibliography{local,references}

\appendix

\section{Scalable Computation of the Input and Output Projectors} \label{jacobian_with_adjoints}

The matrices $\mathbb{E}_\nu[\nabla q^T \nabla q]$, $\mathbb{E}_\nu[qq^T]$ can be approximated in a scalable and efficient manner using sampling and matrix-free randomized linear algebra. The integrals for  $\mathbb{E}_\nu[\nabla q^T \nabla q]$, $\mathbb{E}_\nu[qq^T]$  can be approximated as finite sums via Monte Carlo approximation; given draws $m_i \sim \nu$, one can generate training data $\{(m_i,q_i)\}_{i=1}^{N_\text{data}}$, and approximate the POD basis via Monte Carlo approximation:
\begin{equation} \label{method_of_snapshots}
  \mathbb{E}_\nu[qq^T] \approx \frac{1}{N_\text{data}} \sum_{i=1}^{N_\text{data}} q_iq_i^T \in \mathbb{R}^{d_Q\times d_Q} = \Phi D\Phi^T.
\end{equation}
A calculation of a low rank approximation of $\Phi D\Phi^T$ from samples is referred to as the method of snapshots since the training data $q_i$ are taken to be ``snapshots'' of the outputs. The integral $\mathbb{E}_\nu[\nabla q^T \nabla q]\in \mathbb{R}^{d_M}$  can also be approximated via the action of the matrices on Gaussian random vectors. The integral approximation is obtained via Monte Carlo, for $\omega_M \in \mathbb{R}^{d_M}$:
\begin{align}
  \mathbb{E}_\nu[\nabla q^T \nabla q ]\omega_M &\approx \frac{1}{N_\text{data}}\sum_{i=1}^{N_\text{data}} \nabla q(m_i)^T\nabla q(m_i) \omega_M
\end{align}
This matrix-free computation requires the of the action of operators \linebreak $\nabla q(m_i)\in \mathbb{R}^{d_Q\times d_M}$ and $\nabla q(m_i)^T \in \mathbb{R}^{d_M \times d_Q}$ at various points $m_i \in \mathbb{R}^{d_M}$. The Monte Carlo sum leverage modern computing architecture and use sample parallelism, making the computation for many samples tractable in terms of time and memory.

\subsection{Approximation with randomized linear algebra}

The single pass randomized eigenvalue decomposition algorithms can approximate the $k$ low rank eigenvalue decomposition of a matrix $Q_kD_kQ_k^T \approx A\in \mathbb{R}^{n \times n}$ for $k+p$ matrix vector products, with expected approximation error bounded by
\begin{equation}
  \mathbb{E}_\rho[\|A - Q_kD_kQ_k^T\|] \leq  \bigg(1 + 4 \frac{\sqrt{n(k+p)}}{p-1} \bigg)|d_{k+1}|.
\end{equation}
See \cite{HalkoMartinssonTropp2011,MartinssonTropp2020}, More accuracy can be obtained via double pass or power iteration, but this requires more applications of the operators. Here $\mathbb{E}_\rho$ denotes the expectation taken with respect the the Gaussian measure $\rho$ from which the random matrix $\Omega \in \mathbb{R}^{n + (k+p)}$ is sampled. Higher accuracy in the approximation can be achieved for more applications of the operator (i.e. power iteration, multi-pass methods). Automated procedures such as adaptive range finding can be used to find the rank $k$ such that a specific tolerance is met, i.e. for a given $\epsilon>0$ find $k,p$ such that
\begin{equation}
  \bigg(1 + 4 \frac{\sqrt{n(k+p)}}{p-1} \bigg)|d_{k+1}| \leq \epsilon.
\end{equation}

The matrix approximation is formed via the action of the operator on Gaussian random vectors which is guaranteed to resolve dominant modes of the operator due to concentration of measure, matrix concentration inequalities such as Bernstein or Chernov. This is an inherently scalable process since all of the matrix vector products are independent of one another, i.e. they can be efficiently parallelized and high levels of concurrency can be leveraged. This is in contrast to other matrix-free approximations such as Krylov, where the matrix vector products are inherently serial.

\subsection{The action of $\nabla q$ and $\nabla q^T$ using adjoints}
The implicit derivative of the outputs $q(m) = q(u(m))$ with respect to the the parameters $m$ can be found via a Lagrange multiplier approach to incorporate the implicit dependence on the weak form of the state equation. Since $q(m) \in \mathbb{R}^{d_Q}$, the result is derived for each component:
\begin{equation}
  \mathcal{L}_i(u,m,v) = q(u(m))_i + v_i^TR(u,m),
\end{equation}
or equivalently take all derivatives at once, with the adjoint variable matrix $V \in \mathbb{R}^{d_V \times d_Q}$:
\begin{equation}
  \mathcal{L}(u,m,v) = q(u(m)) + V^TR(u,m).
\end{equation}
A variation of $\mathcal{L}$ with respect to the adjoint variable matrix yields the state equation:
\begin{equation}
  R(u,m) = 0, 
\end{equation}
which is solved for $u$, given a draw $m$. A variation of $\mathcal{L}$ with respect to the state variable $u$ yields the adjoint equation:
\begin{equation}
  \underbrace{\frac{\partial q}{\partial u}}_B + V^T\underbrace{ \frac{\partial R}{\partial u}}_A = 0 
\end{equation}
which is solved for each adjoint variable:
\begin{equation}
  V^T = -BA^{-1}.
\end{equation}
Lastly a variation of $\mathcal{L}$ with respect to the parameters $m$ allows for the formation of the Jacobian:
\begin{equation}
  \nabla_m q + V^T \underbrace{\frac{\partial R}{\partial m}}_C = 0.
\end{equation}
So finally one can compute 
\begin{equation}
  \nabla_m q(u(m)) = BA^{-1}C = \frac{\partial q}{\partial u}\bigg[\frac{\partial R}{\partial u}\bigg]^{-1}\frac{\partial R}{\partial m}.
\end{equation}
To evaluate at a point $m_i$ one must solve the state and adjoint equations first before applying the operator in a direction $\widehat{m} \in \mathbb{R}^{d_M}$:
\begin{equation}
  \nabla_m q(u(m_i))\widehat{m} = BA^{-1}C \widehat{m}.
\end{equation}
The state and adjoint equations only need to be calculated once before the action of the operator in many directions $\widehat{m}$ is formed. Similarly the transpose of the operator can be applied to $q \in \mathbb{R}^{d_Q}$ via
\begin{equation}
  \nabla_m q(u(m))^T = C^TA^{-T}B^T \widehat{q}.
\end{equation}

The key observation about the action of $\nabla q$ and $\nabla q^T$ is that unlike the forward solution of the PDE state equation they do not involve any nonlinear iterations. These operators only involve linearizations of the PDE. These matrices can be efficiently computed and applied over and over to many right hand sides, which makes the cost of evaluating the derivatives via adjoint methods cheaper than the forward solution of the parametric PDE. The expensive part of the derivative evaluation is the solution of the forward nonlinear PDE, which is required for training data generation anyways. Thus one can save computation by generating training data and approximations of derivative subspaces in tandem. The computational cost of computing derivative subspaces of the outputs is marginally less expensive than the evaluation of the nonlinear mapping. 
The costs of the linearizations can be made negligible in regimes where direct solvers is used for the forward mapping, or iterative solvers with preconditioners are used. In the case of a direct solver, the matrix factorization costs can be amortized when applied to many different right hand sides. When an iterative solver with a preconditioner is used, the action of the preconditioner and its transpose can be reused to evaluate the action of $\nabla q$ and $\nabla q^T$.

\section{Numerical Experiment Appendix}
\label{num_exp_appendix}

\subsection{Notes on implementation specifics} \label{appendix_implementation_specifics}

A complete description of parametric PDEs and neural networks can be found in \url{https://github.com/hippylib/hippyflow}. For the convection-diffusion-reaction problem $256$ samples were used in the Monte Carlo approximation for the active subspace projector. For the Helmholtz problem $128$ were used. For both convection-diffusion-reaction and Helmholtz problems $400$ samples were used in the Monte Carlo approximation of the proper orthogonal decomposition. No sampling was required for the KLE projector computation.

The AS and KLE projectors in \texttt{hIPPYflow} are covariance matrix and mass matrix orthogonal respectively. This makes the computation of these projectors consistent in the infinite dimensional limit. Numerical results for the neural network approximation were improved however when these projectors were re-orthogonalized with respect to the identity matrix. Rescaling also improved the numerical results. See \url{https://github.com/hippylib/hippyflow} for more information.

Gradient and Hessian subsampling were used to improve both the performance and computational cost of the inexact Newton CG optimizer. When the number of training data were greater than or equal to $256$, a gradient batch size of $128$ was used and a Hessian batch size of $16$ was used. When the batch size was less than 256, a gradient batch size of $N_\text{data}/4$ was used, and a Hessian batch size of $N_\text{data}/32$ was used. For more information on gradient and Hessian subsampling see \cite{OLearyRoseberryAlgerGhattas2019}.

Unregularized least squares led to better testing accuracy in this problem (probably need surgical regularization that leaves the informed modes unchanged). In order to improve the conditioning of the CG solve in Inexact Newton CG, Levenberg-Marquardt damping of $\gamma = 10^{-3}$ was used. For information on the implementation specifics of the optimizer, see \url{https://github.com/tomoleary/hessianlearn}.

\subsection{Distribution for uncertain parameters $m$} \label{section_matern_prior}
Both problems use a Gaussian probability distribution for the uncertain parameters $m$, with Mat\'{e}rn covariance that is a fractional PDE operator:
\begin{equation} 
  C = (\delta I - \gamma \nabla \cdot (\Theta \nabla))^{-\alpha}.
\end{equation}

The uncertain parameters $m$ are the solution of a linear fractional stochastic PDE \cite{LindgrenRueLindstrom2011}. The choice of $\alpha > d/2$ where $d$ is the spatial dimension of the PDE, makes the covariance trace class, in both problems we take $\alpha = 2$. When the covariance is trace class this guarantees that the uncertain parameters $m$ is $L^2$ integrable, which makes the PDE well-posed. This trace class condition is also typically required for the well-posedness of outer loop inference problems that the surrogate is to be used in. The coefficients $\delta, \gamma>0$ that show up in the fractional PDE control the correlation length as well as the marginal variance of the probability distribution (i.e. the difficulty of the problem). The matrix $\Theta \in \mathbb{R}^{d\times d}$ introduces spatial anisotropy \cite{ChenVillaGhattas2019}. When the PDE problem is discretized the vector representation of $m$ has dimension $d_M$ and we consider $m$ to be a vector in $\mathbb{R}^{d_M}$.

The correlation length for draws from $\nu$ is controlled by the ratio $\delta / \gamma$, and for a fixed correlation length, larger values of $\gamma$ and $\delta$ reduce the marginal variance for distribution.

\subsection{Convection-Diffusion-Reaction PDE Problem}

The volumetric forcing function $\mathbf{f}$ is a Gaussian bump located at $x_1 = 0.7,x_2=0.7$.

\begin{equation}
  \mathbf{f}(x,y) = \max\bigg( 0.5, e^{-25(x_1 - 0.7)^2  -25 (x_2 - 0.7)^2}\bigg)
\end{equation}

\begin{figure}[H]
\center
\includegraphics[width = 0.7\textwidth]{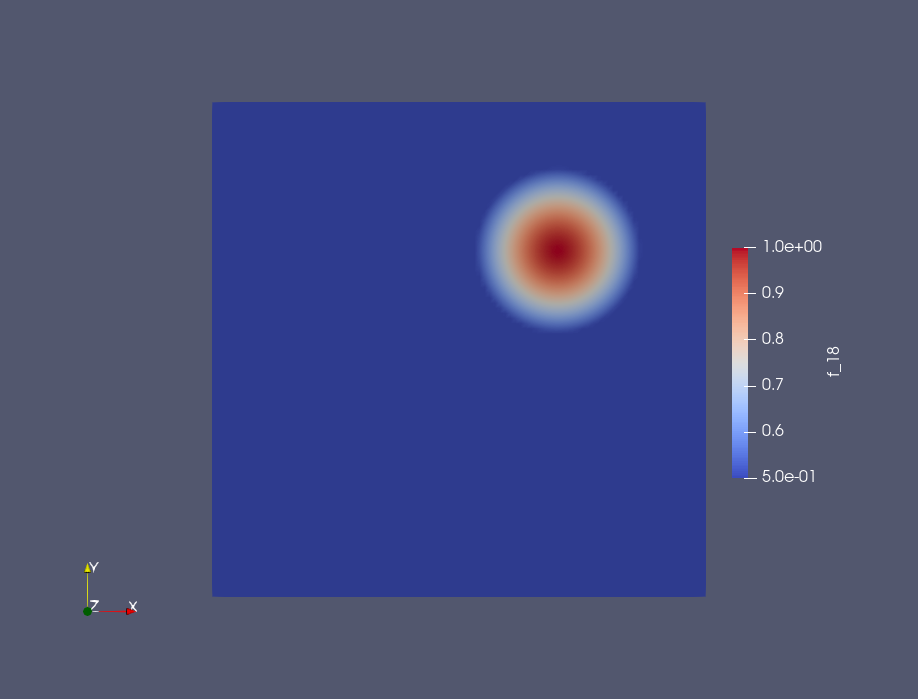}
\caption{Truncated Gaussian blob forcing term}
\end{figure}

The velocity field $\mathbf{v}$ used for each simulation is the solution of a steady-state Navier-Stokes equation with side walls driving the flow:
\begin{subequations}
\begin{align}
  -\frac{1}{Re}\Delta \mathbf{v} + \nabla q + \mathbf{v} \cdot \nabla \mathbf{v} &= 0 \text{ in } \Omega \\ 
  \nabla \cdot \mathbf{v} &= 0 \text{ in } \Omega \\
  \mathbf{v} &= \mathbf{g}  \text{ on } \partial \Omega.
\end{align}
\end{subequations}
The coefficient $Re = 100$ is the Reynolds number, and the Dirichlet term $\mathbf{g}$ is given by $\mathbf{g} = e_2$ on the left wall, $\mathbf{g} = -e_2$ on the right wall and zero everywhere else (see the Advection-Diffusion Bayesian Tutorial in hIPPYlib for more information \cite{VillaPetraGhattas2018}). The velocity field is shown below.

\begin{figure}[H]
\center
\includegraphics[width = 0.7\textwidth]{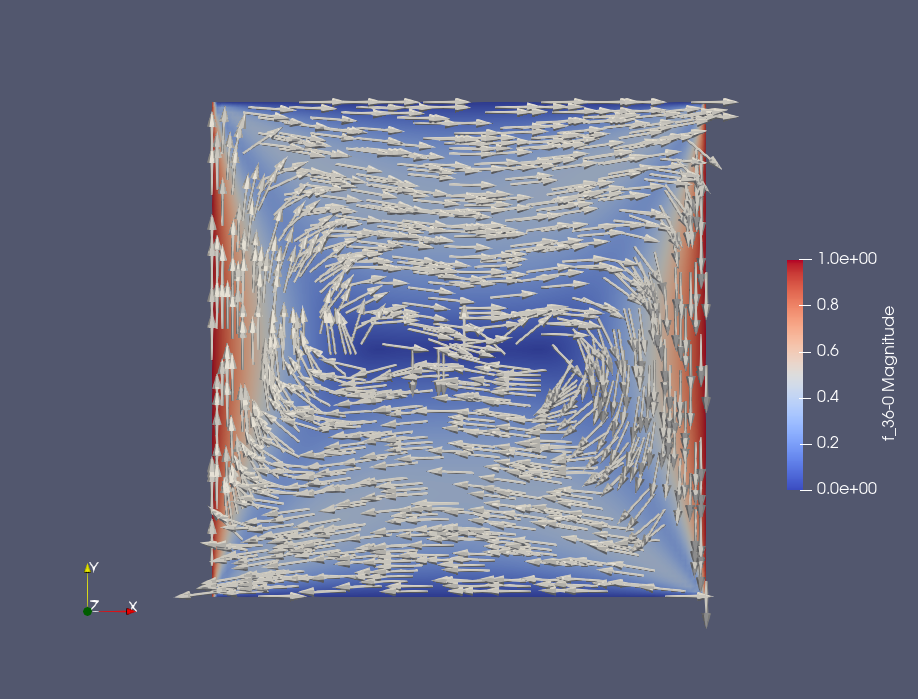}
\caption{Velocity field for convection-diffusion-reaction PDE}
\end{figure}

\end{document}